%% file: pp-pif-weno.tex
\RequirePackage{fix-cm}
\documentclass[smallcondensed]{svjour3}     % onecolumn (ditto)
\pdfoutput=1                                % Useful for making Arxiv happy

\smartqed                  % flush right qed marks, e.g. at end of proof
\usepackage{graphicx}

\usepackage{mathptmx}      % use Times fonts if available on your TeX system

\usepackage[utf8]{inputenc}

\usepackage[english]{babel}

\usepackage{amssymb, amsmath}

\usepackage{color}

\usepackage[numbers, square, comma, sort&compress]{natbib}
\usepackage{hyperref}

\include{styles}

\journalname{Journal of Scientific Computing}

\begin{document}

\title{
An explicit high-order single-stage single-step positivity-preserving finite
difference WENO method for the compressible Euler equations
}
\titlerunning{Positivity-preserving Picard integral formulated finite difference WENO}

\author{David C. Seal \and
        Qi Tang \and
        Zhengfu Xu \and \\
        Andrew J. Christlieb}
\institute{
    David C. Seal \at
    Department of Mathematics \\
    U.S. Naval Academy \\
    121 Blake Road \\
    Annapolis, MD 21402 USA \\
    Tel.: +1(410) 293-6784 \\
    \email{seal@usna.edu} 
\and
    Qi Tang \at
    Department of Mathematical Sciences \\
    Rensselaer Polytechnic Institute \\
%   110 $8^{th}$ Street \\
    Troy, NY 12180, USA
\and
    Zhengfu Xu \at
    Department of Mathematical Sciences \\
    Michigan Technological University \\
%   1400 Townsend Drive \\
    Houghton, MI 49931, USA \\
\and
    Andrew J. Christlieb \at
    Department of Mathematics and
    Department of Electrical and Computer Engineering \\
    Michigan State University \\
    East Lansing, MI 48824, USA \\
}

\date{Received: date / Accepted: date}

% Title information
\maketitle

%% ------------------------------------------------------------------------- %%
%% Abstract
%% ------------------------------------------------------------------------- %%
\begin{abstract}
In this work we construct a high-order, single-stage, single-step positivity-preserving 
method for the compressible Euler equations.  Space is
discretized with the finite difference weighted essentially non-oscillatory
(WENO) method.  Time is discretized through a Lax-Wendroff procedure that is
constructed from the Picard integral formulation (PIF) of the partial differential
equation.  The method can be viewed as a modified flux approach, where
a linear combination of a low- and high-order flux defines the numerical flux
used for a single-step update.  The coefficients of the linear combination are
constructed by solving a simple optimization problem at each time
step.  The high-order flux itself is constructed through the use of
Taylor series and the Cauchy-Kowalewski procedure that incorporates
higher-order terms.  Numerical results in one- and two-dimensions are
presented.

%% Keywords for JSC are defined in the abstract ... 

\end{abstract}
%% ------------------------------------------------------------------------- %%

% \cortext[cor]{Corresponding author.}

%%%%%%%%%%%%%%%%%%%%%%%%%%%%%%%%%
% Input all the sections:
%%%%%%%%%%%%%%%%%%%%%%%%%%%%%%%%

% Note that using \input{FILENAME} is identical to copying and pasting the contents of FILENAME into this location.
% If we were to use \include{FILENAME}, then compilation can be made to go faster, at the cost of needing to run a clear page before and after.

%% ------------------------------------------------------------------------- %%
%% Section: Introduction
%% \label{sec:introduction}
%% \input{Introduction.tex}  
%% ------------------------------------------------------------------------- %%
\section{Introduction} \label{sec:introduction}

The objective of this work is to define a single-stage, single-step finite
difference WENO method that is provably positivity-preserving for the
compressible Euler equations.
%The objective of this work is the development of a new numerical approach to
%the compressible Euler equations based on modified fluxes resulting from the
%Picard integral formulation of hyperbolic conservation laws for systems of
%equations.  
These equations describe the evolution of
density $\rho$, momentum $\rho \vec{u}$ and energy $\E$ of an ideal gas through
\begin{equation}
\label{eqn:euler-eqns-1d}
\left( 
\begin{array}{c}
    \rho \\ \rho \vec{u} \\ \E
    \end{array}
    \right)_{, t}
    + 
    \nabla_{\bf x} \cdot
    \left( 
    \begin{array}{c}
       \rho \vec{u} \\ \rho \vec{u} \otimes \vec{u} + p {\bf I}\\ ( \E + p ) \vec{u}
    \end{array}
    \right)
    = 0.
\end{equation}
%% ------------------------------------------------------------------------- %%
%In Eqn. \eqref{eqn:euler-eqns-1d}, 
The energy $\E$ is related to the primitive 
variables $\rho$, $\vec{u}$ and pressure $p$ 
by the equation of state that we take to be $\E = \frac{p}{ \gamma-1 } + \frac{1}{2}\rho \norm{\vec{u}}^2$.  
The ratio of specific heats is the constant $\gamma$.
Numerical difficulties for solving this system include the
following:
\begin{itemize}
\item Low (1st and 2nd) order methods generally 
suffer from an inordinate amount of numerical diffusion.
However, they are oftentimes more robust, and in some
cases they have provable convergence to the correct entropy solution. 
Historically, 2nd-order schemes \cite{LxW60,Sod78,Harten78,Roe81} 
have been called ``high-resolution'' methods when compared to their 1st-order
counterparts \cite{VonRict50,CoIsRe52,Lax54,God57}.
\item High-order methods \cite{HaEnOsCh87,ShuOsher87,LiuOsherChan94} provide greater accuracy and resolution for much
less overall computational effort.  However, they 
are oftentimes less robust, and do not necessarily have
provable convergence to the correct entropy solution.
\end{itemize}
In this work, we define a high-order conservative finite difference method 
based upon the Picard integral formulation of the PDE.  We make a further
modification to the fluxes and define a numerical scheme that obtains the following properties:
\begin{itemize}
\item High-order accuracy in space ($5^{th}$-order) and time ($3^{rd}$-order).
Our method can be extended to arbitrary order in space or time.
\item A robust scheme that stems from provable positivity preservation 
of the pressure and density.  Numerical results indicate that high-order
accuracy is retained with our positivity-preserving limiter turned on.
\end{itemize}
Our scheme is the first single-stage, single-step numerical method that
simultaneously attains high-order accuracy, with
\emph{provable} positivity preservation.  
When compared to other positivity-preserving schemes,
our method has the following advantages:
\begin{itemize}
    \item In order to retain positivity, we only solve one simple optimization
    problem per time step.
    Unlike positivity-``preserving'' methods that use Runge-Kutta discretizations \cite{XiQiXu14,ChLiTaXu14}, 
    positivity of our solution is guaranteed during the entire simulation
    because we do not have internal stages where the solution can go negative.
    \item Compared to other positivity-preserving schemes
    \cite{ZhangShu11-survey,ZhangShu12}, 
    the addition of the positivity
    preserving limiter introduces none of the additional time step restrictions that are
    often introduced in order to retain positivity.
\end{itemize}
In addition, our method is amenable to adaptive mesh refinement (AMR) technology.
At present, we aim to lay the necessary foundation that would be required to do so.
An in depth investigation of this property is the subject of a future study.

\subsection{An overview of the proposed method}

The Euler equations define a system of hyperbolic conservation laws.  In 1D, such an
equation is given by
\begin{equation}
\label{eqn:1dcons}
  q_{,t} + f(q)_{\!,\,x} = 0,
\end{equation}
where $q(t,x):\mathbb{R}^+\times \mathbb{R} \to \mathbb{R}^m$ is the unknown 
vector of $m$ conserved quantities and $f:\mathbb{R}^m \to \mathbb{R}^m$ is 
a prescribed flux function.  
The conserved variables for the 1D Euler equations are 
$q = \left( \rho, \rho u^1, \E \right)^T$.
A typical finite difference solver for \eqref{eqn:1dcons} 
discretizes space with a uniform grid of $m_x$ equidistant points in $\Omega = [a,b]$,
\begin{equation}
  x_i = a + \Par{i-\frac{1}{2}}\Delta x,
  \qquad \Delta x = \frac{b-a}{m_x},
  \qquad i\in\{1,\dots,m_x\},
\end{equation}
and looks for a pointwise approximation $q_i^n \approx q(t^{n},x_i)$ solution to
hold at discrete time levels $t^n$.  In a conservative finite difference WENO
method, the update of the unknowns is typically defined by
\begin{equation}
\label{eqn:1dpif}
    q^{n+1}_i = q^n_i - \frac{\Dt}{\Dx}\left( {F}^n_{i+\hf}-{F}^n_{i-\hf} \right),
\end{equation}
where the numerical flux ${F}^n_{i\pm\hf}$ is constructed from a
linear combination of the WENO reconstruction procedure applied to stage
values from a Runge-Kutta solver.

In this work, we propose the following procedure:
\begin{enumerate}
\item Construct a high-order approximation to the \emph{time-averaged fluxes} 
\cite{HaEnOsCh87,SeGuCh14}
\begin{equation}
\label{eqn:picard1d-b}
    {F}^n_i := \frac{1}{\dt} \int_{t^n}^{t^{n+1}} f( q(t, x_i ) )\, dt
\end{equation}
at each grid point $x_i$.  In this work, we consider the Taylor
discretization of Eqn. \eqref{eqn:picard1d-b} for conservative finite
difference methods.
\item 
Construct a 
{\bf high-order in space and time} numerical flux $\hat{F}^n_{i-\hf}$ based
upon applying the WENO reconstruction procedure to the time-averaged fluxes
\cite{SeGuCh14}.
%This too be described in \S\ref{sec:review}.
%
\item Replace the flux constructed in Step 2 with 
\begin{align}
\label{eqn:limited-flux}
    \tilde{F}^n_{i-\hf} := \theta^n_{i-\hf}(\hat{F}^n_{i-\hf} - \hat{f}^n_{i-\hf}) + \hat{f}^n_{i-\hf},
\end{align}
where $\theta^n_{i-\hf} \in [0,1]$ is found by solving a single optimization problem, and
$\hat{f}^n_{i-\hf}$ is a low-order flux that guarantees positivity of the
solution.  This procedure is described in \S\ref{sec:1d}.
\item Insert the result of Step 3 into Eqn.  \eqref{eqn:1dpif}, and update the
solution.
\end{enumerate}

Steps 1 and 2 have already been proposed in \cite{SeGuCh14}, where high-order
accuracy is obtained through a flux modification that incorporates the
high-order temporal discretization.  A review of this procedure is 
presented in \S\ref{sec:review}.
Step 3 can be thought of as a further flux modification, where an 
automatic switch adjusts between the high-order non positivity-preserving 
scheme, and a low-order, positivity-preserving scheme.
The original idea is attributed to Harten and Zwas \cite{HaZw72}, but
has since been extended to high-order WENO schemes \cite{XiQiXu14}.
The details of this procedure are presented in \S\ref{sec:1d} and \S\ref{sec:md}.

\section{Background}

The compressible Euler equations have been an object of study ever since the
infancy of numerical methods \cite{VonRict50,CoIsRe52,Lax54,God57,LxW60}.
In recent years, high-order methods have attracted considerable interest
because of their ability to obtain higher accuracy on certain problems with an
equivalent computational cost of a low-order method.
Among many choices of high-order %shock-capturing 
schemes are the classical essentially non-oscillatory 
(ENO) method \cite{HaEnOsCh87}, the extensions to 
finite difference (FD) and finite volume (FV) WENO methods \cite{LiuOsherChan94,JiangShu96,Shu09},
and the discontinuous Galerkin (DG) method \cite{colish89}.
These methods all seek to
simultaneously obtain two properties: retain high-order accuracy in smooth
regions, and capture shocks without introducing spurious oscillations near
discontinuities of the solution.
%have been extensively studies and
%developed because of their robustness and non-oscillatory behavior around 
%shock regions.  
%For a recent review of WENO schemes coupled with Runge-Kutta (RK) time
%stepping, see \cite{Shu09}.
%
One added difficulty with high-order schemes is the necessity of defining and
selecting a high-order time integrator.
%   The method of lines (MOL) approach is the most widely used temporal
%   discretization for high-order schemes.
%   This approach first discretizes a PDE space, and then applies an appropriate
%   time-integrator to the resulting system of ODEs.
%   \change{The MOL technique helps
%   people mainly focus on the high-order spatial discretization while taking use
%   of many well-developed ODE timing stepping such as Runge-Kutta (RK) methods.}
Runge-Kutta methods applied to the method of lines (MOL) approach is the most 
widely used discretization for high-order schemes.  These methods all treat
space and time as separate entities.
%This approach first discretizes a PDE space, and then applies an appropriate
%time-integrator to the resulting system of ODEs.

Over the past decade, there has been a rejuvenation of interest in high-order
single-stage, single-step methods for hyperbolic conservation laws, including the compressible Euler equations. 
All of these methods are typically based upon a Taylor 
temporal discretization that uses the Cauchy-Kowalewski procedure to exchange
temporal for spatial derivatives.  Lax and Wendroff performed this very
procedure in 1960 \cite{LxW60}, and this technique has since been called the Lax-Wendroff
procedure within the numerical analysis community.
Methods for defining a second and higher-order single-step version of Godunov's method
were investigated in the 1980s \cite{Ro87,BeCoTr89,Me90}.
The original high-order ENO method of Harten et. al \cite{HaEnOsCh87} used Taylor series
for their temporal discretization, although most of the attention they have received
is for their emphasis on the spatial discretization.
In 2001, the preliminary definitions for the so-called 
Arbitrary DERivative (ADER) methods \cite{toro2001,ToroTitarev02,TitarevToro02-ader} were put in place.  
Additionally, various FD WENO methods with
Lax-Wendroff time discretizations have been constructed and tested
on the Euler equations \cite{QiuShu03,JiangShuZhang13,SeGuCh14}.  
Recent ADER methods have been defined by Balsara and collaborators for hydrodynamics and
magnetohydrodynamics \cite{balsara09,BaDiMeDuHuXu13}, and have later been extended to an
adaptive mesh refinement (AMR) setting \cite{balsara13}.
Other recent work in single-stage, single-step methods for Euler
equations includes Lax-Wendroff time stepping coupled with DG 
\cite{QiuDumbserShu05,DuMu06,GaDuHiMuCl11}, and high-order 
Lagrangian schemes \cite{LiuChengShu09}.  The present work is an extension of
the Taylor discretization of the Picard integral formulation that uses
finite differences for its spatial discretization \cite{SeGuCh14}, which falls
into this single-stage, single-step class of methods.

%   One overarching goal that all single-step methods share is that they are
%   are easier to fit into an Adaptive Mesh
%   Refinement (AMR) framework, whereas multistage Runge-Kutta (RK) methods require
%   a more complicated treatment for ghost cells and communication for each AMR patch.  

Defining high-order numerical schemes that retain positivity of the solution
for hydrodynamics (or magnetohydrodynamics) simulations is
genuinely a non-trivial task.  This has been an ongoing subject of study even
for low and the so-called ``high-resolution'' schemes 
\cite{EinfMuRoeSj91,EstVil95,EsVi96,PeShu96,TangXu99,Dubroca99,TangXu00,Gallice03}.
%In hydrodynamics or magnetohydrodynamics simulations, to retain positivity of
%the solution is genuinely a non-trivial task, and has been a subject of study
%since the popularity of those high-order schemes.  
All methods that are second or higher order share the 
same disadvantage that without care, they may violate a natural
weak stability condition that the density and pressure need to keep positive,
which is necessary to ensure physical meaningfulness of the solution and hyperbolicity of the mathematical problem.
Of some of the early positivity works, Perthame and Shu propose a general reconstruction approach to obtain a
high-order positivity-preserving finite volume schemes from a low-order scheme
\cite{PeShu96}.  In addition, they prove that the explicit Lax-Friedrichs
scheme is positivity-preserving with a CFL number up to $0.5$.  Later on, a
more general result extends the 
positivity-preserving property to CFL numbers up to $1$ for 
both explicit and implicit Lax-Friedrichs methods
\cite{TangXu00}.  With those building blocks, a
positivity-preserving limiter is proposed for DG schemes
\cite{ZhangShu10,ZhangShu11-euler,ZhangXiaShu12,KoEk13} and FD and FV WENO schemes
\cite{ZhangShu11-survey,ZhangShu12} for the Euler equations.  In
\cite{hu2013}, a flux cut-off limiter is also applied to FD WENO schemes to
retain positivity.  
In addition to gas dynamics, plasma physics is another
area where retaining positivity of numerical solutions is critical, and
therefore has seen recent attention in the literature \cite{RoSe11,GuChHi14}.
For example, collision operators for Vlasov equations require a positive
distribution function in order to avoid creating artificial singularities.
%positivity-preserving limiters have been used
%in \cite{RoSe11,QiuShu11,GuChHi14} to obtain this property.  
%In \cite{ZhangShu11-survey}, the authors also  and references therein for recent work.
%
%All of these methods are either low-order, use multistage Runge-Kutta
%temporal discretizations, or are not provably positivity-preserving.
%Our contribution is to define a high-order, single-stage, single-step method 
%for the Euler equations that is provably positivity-preserving.

Our method is based upon a parameterized flux limiter that can be dated back
to at least 1972 with the work of Harten and Zwas \cite{HaZw72}.  There, the
authors propose a second-order shock-capturing
\emph{self adjusting hybrid scheme} through a simple linear combination of
low- and high-order fluxes that is identical to Eqn. \eqref{eqn:limited-flux}.  The
original idea is to combine a ``high-order" flux with a
first-order flux such that it has better accuracy in smooth region and
produces a smooth profile around shock regions.  A similar approach called
\emph{flux-corrected transport} is proposed by Boris and Book
\cite{Boris197338,Book1975248,Boris1976397,BoBoZa81}, where the purpose of limiting the
high-order flux is to control overshoots and undershoots around shock
regions.  Sod performs an extensive review of these and other classical
finite difference methods in a classical paper \cite{Sod78}.
Xu and Lian recently extend this work to WENO methods that maintain the
maximum-principle-preserving property for scalar hyperbolic conservation laws
with the so-called ``parametrized flux limiter'' \cite{xu2013,liang2014parametrized}.
Later on, these limiters are applied to FD WENO schemes on
rectangular meshes \cite{XiQiXu13} and FV WENO schemes on triangular meshes
\cite{ChLiTaXu14} for the Euler equations.  These limiters are also
applied to magnetohydrodynamics with a constrained transport framework \cite{tang14}.  
The basic idea for all of these methods is the same: modify the
high-order non positivity-preserving numerical flux by a linear combination of a
low- and high-order flux in order to retain positivity of the solution.  The
modification is carefully designed so that the high-order accuracy
remains.

The purpose of this work is to define a single-stage, single-step finite
difference WENO method that is provably positivity-preserving for the
compressible Euler equations.  
Of the various finite difference schemes constructed from the
Picard integral formulation \cite{SeGuCh14}, 
we begin with the Taylor discretization,
and then apply recently developed flux limiters
\cite{XiQiXu13,ChLiTaXu14} in order to retain positivity of the solution.
One advantage of the chosen limiter is that positivity is preserved without
introducing additional time step restrictions,  however, our primary
contribution is that the present method is the first
scheme to simultaneously be high-order, single-stage, single-step and
have provable positivity preservation.

The outline of this paper is as follows.
In \S\ref{sec:review}, we briefly review the high-order finite difference WENO method 
that is based upon the Picard integral formulation of the PDE with a Taylor
temporal discretization \cite{SeGuCh14}.
In \S\ref{sec:1d} and \S\ref{sec:md}, we present the positivity-preserving 
limiter for PIF-WENO schemes applied to the compressible Euler system in 
single and multiple dimensions.
Numerical examples of the positivity-preserving PIF-WENO scheme applied to the
problems with low density and low pressure is provided in
\S\ref{sec:numerical-results}. 
Finally, conclusions and future work are given in \S\ref{sec:conclusions}.
%List of references to be included:
%\begin{enumerate}
%\item Lax-Wendroff paper \cite{LxW60}.  
%\item Classical WENO method \cite{HaEnOsCh87}.  Review paper with good references \cite{Shu09}.
%\item Finite difference WENO papers with Lax-Wendroff discretization \cite{QiuShu03,JiangShuZhang13,SeGuCh14}.
%\item Some of Zhengfu's papers \cite{XiQiXu13,XiQiXu14}.
%\item Proof of positivity preservation of FV schemes for Euler equations \cite{PeShu96}.
%\item positivity-preserving methods for Euler \cite{EinfMuRoeSj91,EstVil95,TangXu99,TangXu00,ZhangShu10,ZhangShu11-survey,ZhangShu11-euler,ZhangXiaShu12,ZhangShu12,CaCrGoPe13,KoEk13}
%\item positivity-preserving methods for MHD \cite{Balsara12}.
%\item ADER schemes for Euler \cite{BaDiMeDuHuXu13}
%\item High-order single-step methods for hyperbolic conservation laws
%\cite{Ro87,BeCoTr89,Me90,ToTi06,GaDuHiMuCl11,Li12,LiuChengShu09}
%\item Theory papers \cite{Tadmor86,HaLaLeMo98}.
%\item Random entropy papers (to get provable convergence? \cite{GuPo14} )
%\item Review papers \cite{Sod78,WoodwardColella84}.
%\item Some Euler methods \cite{ReFlaShep03} (nice description of double Mach-reflection problem).
%\end{enumerate}

\section{A single-stage single-step finite difference WENO method} 
\label{sec:review}

The numerical method that is the subject of this work is based upon the Taylor
discretization of the Picard integral formulation of Euler's
equations, which is one of the many methods developed in \cite{SeGuCh14}.
%The numerical method that is the subject of this work
%is based upon the Taylor discretization
%of the Picard integral formulation \cite{SeGuCh14}.
Our focus is on the finite
difference WENO method based on a Taylor discretization of the time-averaged
fluxes because it easily lends itself to the
positivity-preserving limiters that are presented in \S\ref{sec:1d}.
In this section, we review the minimal 
details presented in \cite{SeGuCh14} that are necessary to reproduce the
present work.  In addition, this section serves to set the notation that is
used in upcoming sections.

In two dimensions, a hyperbolic conservation law is defined by a flux function with two components,
\begin{align}
\label{eqn:conslaw2}
    q_{t} + f(q)_{,x} + g(q)_{,y} = 0,
\end{align}
where $q(t,x,y):\R^+\times \R^2 \to \R^m$ is the vector of conserved variables,
and $f,g :\R^m \to \R^m$ are the two components of the flux function.
The Euler equations are an example of a set of equations from this class of problems.

Formal integration  of \eqref{eqn:conslaw2} in time over $t \in [t^n, t^{n+1}]$
defines the 2D \emph{Picard integral formulation} \cite{SeGuCh14} as
\begin{subequations}
\begin{equation}\label{eqn:picard2d-a}
    q^{n+1} = q^n - \dt \left( F^n( x, y ) \right)_{,x} - \dt \left( G^n( x, y) \right)_{,y},
\end{equation}
where the \emph{time-averaged flux} is defined as
\begin{equation}\label{eqn:picard2d-b}
    {F}^n( x, y) := \frac{1}{\dt} \int_{t^n}^{t^{n+1}} f(q(t,x,y))\, dt, \quad 
    {G}^n( x, y) := \frac{1}{\dt} \int_{t^n}^{t^{n+1}} g(q(t,x,y))\, dt.
\end{equation}
\end{subequations}
The basic idea of the Picard integral formulation of WENO (PIF-WENO) \cite{SeGuCh14}
is to first approximate the time-averaged fluxes \eqref{eqn:picard1d-b} at each grid point using
some temporal discretization,
and then approximate spatial derivatives in \eqref{eqn:picard2d-a}
by applying WENO reconstruction to the resulting time-averaged fluxes.
In this work, we approximate Eqn. \eqref{eqn:picard2d-a} with the finite difference WENO method, 
and we use a third-order Taylor discretization for \eqref{eqn:picard2d-b}.
We remark that the positivity-preserving limiter proposed in \S \ref{sec:1d} can be generally applied to any form of
the Picard integral formulation, including Runge-Kutta time discretizations.

Given a domain $\Omega = [a_x,b_x] \times [a_y, b_y]$, a finite difference approximation
seeks pointwise approximations $q^n_{i,j} \approx q\left( t^n, x_i, y_j \right)$
to hold at each
\begin{subequations}
\begin{align}
  x_i &= a_x + \Par{i-\frac{1}{2}}\Delta x,
  \qquad &\Delta x = \frac{b_x-a_x}{m_x},
  \qquad &i\in\{1,\dots,m_x\}, && \quad \\
  y_j &= a_y + \Par{j-\frac{1}{2}}\Delta y,
  \qquad &\Delta y = \frac{b_y-a_y}{m_y},
  \qquad &j\in\{1,\dots,m_y\},&
\end{align}
\end{subequations}
for discrete values of $t = t^n$.  The 2D PIF-WENO scheme \cite{SeGuCh14} solves 
Eqn. \eqref{eqn:conslaw2} with a conservative form
\begin{align}
\label{eqn:pif-2d}
    q^{n+1}_{i,j} = q^n_{i,j} - 
        \lambda_x\left( \hat{F}^n_{i+\hf,j}-\hat{F}^n_{i-\hf,j} \right) 
      - \lambda_y\left( \hat{G}^n_{i,j+\hf}-\hat{G}^n_{i,j-\hf} \right),
\end{align}
where $\lambda_x = \Dt/\Dx$, $\lambda_y = \Dt/\Dy$, and
$\hat{F}^n_{i\pm \hf, j}$ and $\hat{G}^n_{i, j\pm \hf}$ are high-order fluxes obtained
by applying the classical WENO reconstruction to the time-averaged fluxes
in place of a typical ``frozen-in-time'' approximation to the fluxes.
This requires a total of two steps: construct a time-averaged flux, followed by
performing a WENO reconstruction to the resulting modified fluxes.

%\subsection{Construction of the time-averaged flux}
We first define numerical time-averaged fluxes at each grid point $(x_i, y_j)$ through Taylor expansions.
After taking temporal derivatives of $f$ and $g$, we integrate the resulting
Taylor polynomials over $[t^n,t^{n+1}]$ to yield
\begin{subequations}
\begin{align}
\label{eqn:2D_system.TI-F}
  F_T^n(x,y) := f( q(t^n,x,y) ) 
    + \frac{\dt}{2!}   \frac{df}{dt} ( q(t^n,x,y) ) 
    + \frac{\dt^2}{3!} \frac{d^2f}{dt^2}( q(t^n,x,y) )\\
\label{eqn:2D_system.TI-G}
  G_T^n(x,y) := g( q(t^n,x,y) ) 
    + \frac{\dt}{2!}   \frac{dg}{dt} ( q(t^n,x,y) ) 
    + \frac{\dt^2}{3!} \frac{d^2g}{dt^2}( q(t^n,x,y) ).
%     {F_T}^n(x,y) := 
%   \left| f + \frac{\dt}{2!} \frac{df}{dt} + \frac{\dt^2}{3!} \frac{d^2f}{dt^2} \right|_{q(t^n,x,y)} 
%       = F^n + \BigOh(\dt^3); \\
\end{align}
\end{subequations}
The temporal derivatives that appear can be found via the Cauchy-Kowalewski procedure.
For example, the first two time derivatives of the first component of the flux function are
given by
%For example, we have
\begin{subequations}
\begin{equation}
\label{eqn:time-derivs-2df-a}
   \frac{df}{dt}     = -\pd{f}{q} \cdot \left( f_{\!,\,x} + g_{\!,\,y} \right),
\end{equation}
and
\begin{equation}
\label{eqn:time-derivs-2df-b}
   \frac{d^2f}{dt^2} = \pdn{2}{f}{q} \cdot \bigl( f_{,x} + g_{,y}\,, f_{,x} + g_{,y} \bigr) 
                      + \pd{f}{q} \cdot \left( -f_{,x} - g_{,y} \right)_{,t}.
\end{equation}
\end{subequations}
%for the first component of the flux function.
The last time derivative can be further simplified to
\begin{align}
\label{eqn:temporal-derivs-2dfpg}
-\left( f_{,x} + g_{,y} \right)_{,t} =    
&  \pdn{2}{f}{q} \cdot \left( q_{,x}, f_{,x} + g_{,y} \right) + \pd{f}{q} \cdot \left( f_{,xx} + g_{,xy} \right)
+ \\ \nonumber 
&  \pdn{2}{g}{q} \cdot \left( q_{,y}, f_{,x} + g_{,y} \right) + \pd{g}{q} \cdot \left( f_{,xy} + g_{,yy} \right).
\end{align}
Temporal derivatives of $g$ have a similar structure, and can be found in \cite{SeGuCh14}.
%------------------------------------------------------------------------------

%------------------------------------------------------------------------------
We approximate each $\partial_{x}, \partial_{xx}$ and $\partial_{y}, \partial_{yy}$ 
in \eqref{eqn:time-derivs-2df-a} and \eqref{eqn:time-derivs-2df-b}
by applying the 5-point finite difference formulae
\begin{subequations}
\label{eqn:u-deriv}
\begin{align}
\label{eqn:ux}
u_{i,\, x} &:= \frac{1}{12 \dx} \left( 
        u_{i-2} - 8  u_{i-1} 
      + 8  u_{i+1} - u_{i+2} 
  \right) =
  u_{\, ,\, x}( x_i ) + \BigOh\left( \dx^4 \right)
  \\
\label{eqn:uxx}
u_{i,\, xx} &:= \frac{1}{12 \dx^2} \left( 
      - u_{i-2} 
      + 16 u_{i-1} 
      - 30 u_{i}
      + 16 u_{i+1} - u_{i+2} 
  \right)
  = u_{\, ,\, xx}( x_i ) + \BigOh\left( \dx^4 \right)
\end{align}
\end{subequations}
in each direction.
In order to retain a compact stencil, we compute the cross 
derivatives $\partial_{xy}$ with a second-order approximation
\begin{equation}
u_{ij,\, xy} := \frac{1}{4 \dx \dy } 
    \left(  
        u_{i+1,\, j+1} - u_{i-1,\, j+1} - u_{i+1,\, j-1} + u_{i-1,\, j-1}
    \right), % + \BigOh( \dx^2, \dy^2 )
\end{equation}
which is sufficient to retain third-order accuracy in time.
After defining these higher derivatives, we define numerical fluxes by
$F^n_{i,j} := F^n_T( x_i, y_j )$ and 
$G^n_{i,j} := G^n_T( x_i, y_j )$.
We then apply WENO reconstruction in a dimension by
dimension fashion to each component of the flux
to construct interface values $F^n_{i\pm \hf, j}$ and $G^n_{i, j\pm \hf}$.
The complete description of this process can be found in \cite{SeGuCh14}.

\begin{rmk} \label{rmk:pif}
The Picard integral formulation sets up a discretization for the fluxes, and not the conserved variables.
\end{rmk}
The significance of this remark is that further flux modifications can be 
incorporated into the Picard integral formulation.
Previous finite difference WENO methods with Lax-Wendroff type time discretizations
(e.g. \cite{QiuShu03,JiangShuZhang13}) rely on Taylor expansions of the \emph{conserved variables}, and
not the fluxes; the Taylor discretization of the Picard integral formulation computes Taylor
expansions of the fluxes, and not the conserved variables.  In \cite{QiuShu03}, conservation of mass
comes from the fact that higher derivatives of the conserved variables are computed with a central stencil.
In our scheme, we directly discretize the fluxes, and are automatically mass conservative because we insert the result
into the WENO reconstruction procedure.
Because ours is an operation on the fluxes, we have the opportunity to consider further flux modifications.
In this work, we further modify the fluxes to obtain 
a provably positivity-preserving method for Euler equations, which we now describe.

\section{The positivity-preserving method: the 1D case} \label{sec:1d}
%\input{1dcase}

%In this section, we apply the positivity-preserving limiter to 
%the Taylor discretization of the 1D PIF-WENO scheme presented in
%\S\ref{sec:review}.
We begin with the 1D formulation of the proposed positivity-preserving scheme.
Recall that the update for the vector
of conserved variables is given by Eqn. \eqref{eqn:1dpif} for the 1D
conservation law defined in \eqref{eqn:1dcons}.
We consider a numerical flux $\hat{F}^n_{i-\hf}$ 
that is high order accurate in
time (and space) that is constructed from the 1D Taylor discretization of the 
Picard integral formulation (PIF) of FD-WENO \cite{SeGuCh14},
and we consider a low-order flux 
$\hat{f}^n_{i-\hf}$  that is constructed from the Lax-Friedrichs scheme (that
is provably positivity-preserving \cite{PeShu96}).
Both fluxes are constructed by looking at the solution $q^n$ at time level
$t^n$.

We propose modifying the high-order flux by
\begin{equation}
\label{eqn:flux}
    \tilde{F}^n_{i-\hf} := \theta^n_{i-\hf}(\hat{F}^n_{i-\hf} - \hat{f}^n_{i-\hf}) + \hat{f}^n_{i-\hf},
\end{equation}
where a simple optimization problem is solved for the \emph{limiting parameter} $\theta^n_{i-\hf} \in [0,1]$ at each time step.

We observe that if $\theta^n_{i-\hf} = 0$, the 
scheme reduces to the first-order Lax-Friedrichs scheme, which is positivity
preserving, and therefore it is always possible to find a value that retains
positive density and pressure.
If $\theta^n_{i-\hf} = 1.0$, the scheme reduces
to the high-order scheme, but does not guarantee positivity of the numerical solution.  
In order to retain high-order accuracy, we would like
to choose $\theta^n_{i-\hf}$ as close to $1.0$ as possible without violating 
positivity of the density and pressure.

The positivity-preserving Algorithm we outline in this section follows a two step procedure: 
i) guarantee positivity of the density, and then ii) guarantee positivity of the pressure.
The details of this procedure are spelled out in the following subsections.

\subsection{Step 1: Maintain positivity of the density}

This discussion focuses on the first component of the modified flux
\begin{equation}
	\tilde{f}^{n,\rho}_{i+\hf} := \theta^n_{i+\hf}( \hat{F}^{n,\rho}_{i+\hf} - \hat{f}^{n,\rho}_{i+\hf}) + \hat{f}^{n,\rho}_{i+\hf},
\end{equation}
where $\hat{f}^{n,\rho}$ is the first component of the low-order flux $\hat{f}^n$, 
and $\hat{F}^{n,\rho}$ is the first component of the high-order flux $\hat{F}^n$.
In this step, we must assume that the density ${\rho}^{n}_i > 0$ is positive
at the current time.
We further define the low-order update for the density as 
\begin{equation*}
    \hat{\rho}^{n+1}_i := {\rho}_i^n - \lambda \left( \hat{f}^{n,\rho}_{i+\hf} - \hat{f}^{n,\rho}_{i-\hf} \right),
\end{equation*}
and define a numerical lower bound of the high-order updated density $\rho^{n+1}$ as 
$\epsilon^{n+1}_\rho : = \min ( \min_i \left(\hat{\rho}^{n+1}_i \right
),\epsilon_0)$.
The use of $\epsilon_0 > 0$ guarantees finite wave speeds, because the sound speed
$c := \sqrt{\gamma p / \rho}$ goes to infinity as $\rho \to 0$.
In our simulations, we take $\epsilon_0 = 10^{-13}$, which is consistent with recent high-order
positivity-preserving work \cite{ZhangShu10}.
%
%The purpose of $\epsilon_0$ is a common assumption that we will make about the behavior of the 
%exact solution \cite{ZhangShu10}.
%The wave speeds for the 1D Euler equations are $\left\{ u-c, u, u+c \right\}$, where the sound speed is 
%$c := \sqrt{\gamma p / \rho}$, and the velocity $u$ can be extracted from the vector of unknowns $q = \left( \rho, \rho u, \E\right)$.
%Since the wave speeds of the system involve factors of $1/\rho$, to avoid the potential numerical difficulty of an 
%infinite wave speed that would arise when $\rho$ approaches $0$,
%we assume both density and pressure  have the same positive lower bound $\epsilon_0$ in the exact solutions.
%In the simulation, we take $\epsilon_0 = 10^{-13}$, which is consistent with recent high-order
%positivity preserving work \cite{ZhangShu10}.
%
Thanks to the positivity of the low-order flux \cite{PeShu96}, we observe that $\epsilon^{n+1}_\rho > 0$.

After the low- and high-order fluxes have been computed, the update for the density
at a single grid point $x_i$ only depends on two values of $\theta^n_{i\pm \hf}$ through
\begin{equation*}
    {\rho}_i^{n+1} \left(\theta^n_{i-\hf},\, \theta^n_{i+\hf} \right) = {\rho}_i^n - \lambda \left( \tilde{f}^{n,\rho}_{i+\hf} - \tilde{f}^{n,\rho}_{i-\hf} \right), \quad \lambda = \frac{\Dt}{\Dx}, \quad i \in \left\{ 1, 2, \dots, m_x \right\}.
\end{equation*}
This, and each of the conserved variables are \emph{linear functions} with respect to the 
variable $\left( \theta^n_{i-\hf}, \theta^n_{i+\hf} \right) \in [0,1]^2$.  
To preserve the positivity of the density $\rho^{n+1}$, 
we want to guarantee $\rho^{n+1}_i \geq \epsilon^{n+1}_\rho$.
Therefore, we seek bounds $\Lambda^\rho_{\pm\hf, I_i}$ 
such that whenever
\begin{equation*}
	\left(\theta^n_{i-\hf}, \theta^n_{i+\hf} \right) \in \left[0,\, \Lambda^\rho_{-\hf, I_i} \right] \times 
		\left[0,\, \Lambda^\rho_{+\hf, I_i} \right] \subseteq [0, 1]^2,
\end{equation*}
we have
\begin{align}
\label{liu3}
    {\rho}_i^{n+1} \left(\theta^n_{i-\hf},\theta^n_{i+\hf} \right) = {\rho}_i^n - \lambda \left( \tilde{f}^{n,\rho}_{i+\hf} - \tilde{f}^{n,\rho}_{i-\hf} \right) \geq \epsilon^{n+1}_\rho.
\end{align}
The purpose of defining such a set is that in Step 2 of \S\ref{subsec:step2-1d},
we will further limit the fluxes to maintain positivity of the pressure.

%To continue, we consider the low-order update for the density.  We assume that ${\rho}^{n}_i$ is positive, and we define
%the low-order update for the density as $\Gamma_i := {\rho}_i^n - \lambda ( \hat{f}^{n,\rho}_{i+\hf} - \hat{f}^{n,\rho}_{i-\hf} )$.
%We observe $\Gamma_i \geq 0$ because of the positivity of low-order flux \cite{PeShu96}.  %$\hat{f}_{i+\hf}$.
We insert the definition of $\hat{\rho}^{n+1}_i$ into Eqn. \eqref{liu3} to see
\begin{equation}
\label{inequ1}
    \hat{\rho}^{n+1}_i -  \lambda\left[ 
        \theta^n_{i+\hf} \left( \hat{F}^{n,\rho}_{i+\hf} - \hat{f}^{n,\rho}_{i+\hf} \right) 
      - \theta^n_{i-\hf} \left( \hat{F}^{n,\rho}_{i-\hf} - \hat{f}^{n,\rho}_{i-\hf} \right)  \right] \geq \epsilon^{n+1}_\rho,
\end{equation}
which is equivalent to
\begin{align}
\label{inequ2}
 	 \theta^n_{i-\hf} \Delta{f}_{i-\hf}
-    \theta^n_{i+\hf} \Delta{f}_{i+\hf}   \geq \epsilon^{n+1}_\rho - \hat{\rho}^{n+1}_i,
\end{align}
where $\Delta f_{i-\hf} :=  \lambda ( \hat{F}^{n,\rho}_{i-\hf} -  \hat{f}^{n,\rho}_{i-\hf})$ is a measure of the 
deviation of the high- and low-order fluxes.
Note that the right hand side satisfies 
$\epsilon^{n+1}_\rho - \hat{\rho}^{n+1}_i \le 0$, 
and therefore there is at least one solution that can be found for Eqn. \eqref{inequ2} (namely $\theta^n_{i-\hf} = \theta^n_{i+\hf} = 0$).
%(because of the definition of $\epsilon^{n+1}_\rho$.

We determine
bounds on $\Lambda^\rho_{\pm\hf, I_i}$ through a case-by-case discussion 
based on the signs of $\Delta f_{i-\hf}$ and $\Delta f_{i+\hf}$.  This
analysis has already been
performed for single \cite{xu2013} and multidimensional \cite{liang2014parametrized} scalar problems.
There are a total of four cases of Eqn. \eqref{inequ2} to consider:
\begin{itemize}
  \item \underline{Case 1.} If $\Delta f_{i-\hf} \ge 0$ and $\Delta f_{i+\hf} \le 0$, then we set
\begin{equation*}
	\left( \Lambda^\rho_{-\hf, I_i}, \Lambda^\rho_{+\hf, I_i} \right) := (1,1).
\end{equation*}
  \item \underline{Case 2.} If $\Delta f_{i-\hf} \ge 0$ and $\Delta f_{i+\hf} > 0$, then we define
\begin{equation*}
	\left( \Lambda^\rho_{-\hf, I_i}, \Lambda^\rho_{+\hf, I_i} \right) := \left(1,\min \left(1, \frac{\epsilon^{n+1}_\rho - \hat{\rho}^{n+1}_i}{- \Delta f_{i+\hf}}\right)\right).
\end{equation*}
  \item \underline{Case 3.} If $\Delta f_{i-\hf} < 0$ and $\Delta f_{i+\hf} \le 0$, then we set
\begin{equation*}
	\left( \Lambda^\rho_{-\hf, I_i}, \Lambda^\rho_{+\hf, I_i} \right) := \left(\min \left(1, \frac{\epsilon^{n+1}_\rho - \hat{\rho}^{n+1}_i}{ \Delta f_{i-\hf}}\right),1 \right).
\end{equation*}

  \item \underline{Case 4.} If $\Delta f_{i-\hf} < 0$ and $\Delta f_{i+\hf} >0$, there are two sub-cases to consider.
  \begin{itemize}
    \item  \underline{Case 4a.} If the inequality $\eqref{inequ2}$ is satisfied with $(\theta^n_{i-\hf}, \theta^n_{i+\hf}) = (1, 1)$ then
    we set
    \begin{equation*}
    \left(\Lambda^\rho_{-\hf, I_i}, \Lambda^\rho_{+\hf, I_i} \right) := \left(1, 1 \right).
	\end{equation*}
    \item  \underline{Case 4b.} Otherwise, we choose
    \begin{equation*}
    \left(\Lambda^\rho_{-\hf, I_i}, \Lambda^\rho_{+\hf, I_i} \right) := \left(\frac{\epsilon^{n+1}_\rho - \hat{\rho}^{n+1}_i}{\Delta f_{i-\hf} - \Delta f_{i+\hf}},\frac{ \epsilon^{n+1}_\rho - \hat{\rho}^{n+1}_i}{ \Delta f_{i-\hf} -  \Delta f_{i+\hf}} \right).
    \end{equation*}
  \end{itemize}
  \end{itemize}
  
After considering each of the above cases at each grid element $x_i$, we define the following set
\begin{equation}
    S^\rho_i := \left[ 0, \Lambda^\rho_{-\hf, I_i} \right] \times \left[ 0, \Lambda^\rho_{+\hf, I_i} \right].
\end{equation}
{The obtained set has the property that $\rho^{n+1}_i(\theta^n_{i-\hf},\theta^n_{i+\hf}) \geq \epsilon_\rho^{n+1}$
for any $ (\theta^n_{i-\hf}, \theta^n_{i+\hf}) \in S^\rho_i$.}

\subsection{Step 2: Maintain positivity of the pressure}
\label{subsec:step2-1d}

The second step focuses on the pressure $p^{n+1}_i$.
We begin with the following Lemma, which has already been used in the past \cite{ZhangShu10,XiQiXu14}.
\begin{lem} \label{lem:concave-press}
The pressure function satisfies 
\begin{equation*}
%	p\left( \alpha (\theta^1_{i-\hf}, \theta^1_{i+\hf} ) + (1-\alpha) (\theta^2_{i-\hf}, \theta^2_{i+\hf} ) \right)
%	\geq \alpha p\left( (\theta^1_{i-\hf}, \theta^1_{i+\hf} ) \right)
%	+ (1- \alpha) p\left( (\theta^2_{i-\hf}, \theta^2_{i+\hf} ) \right)
    p\left( q^n_i \left( \alpha \overrightarrow{\theta}^1 + (1-\alpha) \overrightarrow{\theta}^2 \right) \right)
    \geq \alpha   p\left( q^n_i \left( \overrightarrow{\theta}^1 \right) \right)
    + (1- \alpha) p\left( q^n_i \left( \overrightarrow{\theta}^2 \right) \right)
\end{equation*}
%for any $\alpha \in [0,1]$ and $\overrightarrow{\theta}^1, \overrightarrow{\theta}^2 \in [0, 1]^2$.
for any $\alpha \in [0,1]$ and $\overrightarrow{\theta}^1, \overrightarrow{\theta}^2 \in S^\rho_i$.
\end{lem}
\begin{proof}
Provided $\rho > 0$, the pressure function
\begin{equation*}
    p(q) := (\gamma - 1) \left( {\En} -\frac{ \| \rho {\bf u }\|^2}{2 \rho} \right)
\end{equation*}
is concave with respect to the conserved variables 
$q = (\rho, \rho {\bf u}, \E)$ \cite{TangXu99,ZhangShu10,XiQiXu14}.
By definition of the limiting parameter, each of the conserved variables are \emph{linear functions} of
$\overrightarrow{\theta}$, which means 
\begin{equation*}
    q^n_i       \left( \alpha \overrightarrow{\theta}^1 + (1-\alpha) \overrightarrow{\theta}^2 \right) =
    \alpha q^n_i\left( \overrightarrow{\theta}^1\right) + (1-\alpha) q^n_i\left(\overrightarrow{\theta}^2 \right).
\end{equation*}
Together, and as a result of the construction in Step 1, these imply
\begin{align*}
    p\left(q^n_i\left( \alpha \overrightarrow{\theta}^1 + (1-\alpha) \overrightarrow{\theta}^2 \right) \right) & =
    p\left(\alpha q^n_i\left(\overrightarrow{\theta}^1\right) + (1-\alpha) q^n_i\left(\overrightarrow{\theta}^2 \right)\right), \\
    & \geq      \alpha p\left( q^n_i\left( \overrightarrow{\theta}^1 \right) \right) + 
            (1- \alpha) p\left( q^n_i\left( \overrightarrow{\theta}^2 \right) \right).
\end{align*}
% The last inequality uses the concavity of the pressure function.
\qed
%Let $\alpha \in [0,1]$, and $\overrightarrow{\theta}^1, \overrightarrow{\theta}^2 \in S^\rho_i$.
%Then observe that $\rho_i( \overrightarrow{\theta}^1 ), \rho_i( \overrightarrow{\theta}^2 ) \geq 0$ thanks to 
%the definition of $S^\rho_i$ in Step 1, and moreover, each of the conserved variables are \emph{linear functions}
%of $\overrightarrow{\theta}$.  Together, this implies
%\begin{align}
%p\left( \alpha \overrightarrow{\theta}^1 + (1-\alpha) \overrightarrow{\theta}^2 \right)
%= 	\alpha \E\left( \overrightarrow{\theta}^1 \right)
%	+ (1- \alpha) \E\left( \overrightarrow{\theta}^2 \right) 
%	- \frac{1}{2} \frac{}{}
%\end{align}
%because $p$ is a concave function of the conserved variables, provided $\rho_i \geq 0$ \cite{TangXu99,ZhangShu10,XiQiXu14}.
\end{proof}
We define
$p_i\left( \overrightarrow{\theta} \right) := p\left( q^n_i \left( \overrightarrow{\theta} \right) \right)$
for any $\overrightarrow{\theta} \in [0,1]^2$ in order to simplify the notation for the ensuing discussion.

If we use $\hat{p}^{n+1}$ to denote the low-order pressure solved by the flux $\hat{f}^n$,
we can similarly define a numerical lower bound for the pressure as $\epsilon_p^{n+1}: = \min(\min_i\left(\hat{p}_i\right),\epsilon_0)$.
Next, we consider the following subset
\begin{equation}
S^p_i := \left\{ (\theta^n_{i-\hf}, \theta^n_{i+\hf}) \in S^\rho_i :
p_i\left( \theta^n_{i-\hf},\theta^n_{i+\hf} \right) \ge \epsilon_p^{n+1} \right\}
\subseteq S^\rho_i,
\end{equation}
and we observe that $S^p_i$ is convex, thanks to the result of Lemma \ref{lem:concave-press}.
We do not attempt to find the entire boundary of $S^p_i$
because that would be computationally intractable.
Instead, we define a single rectangle $R^{\rho,p}_i$ inside of $S^p_i$ that
define bounds on the limiting parameters.

To do this, we 
consider finitely many points on the boundary of $S^p_i$.
To begin, consider the four vertices of $S_i^\rho$ denoted by
$
A^{k_1,k_2} := (k_1 \Lambda^\rho_{-\hf,i}, k_2 \Lambda^\rho_{+\hf,i}),
$
with $k_1$, $k_2$ being 0 or 1.
%, and the vertices of $S_i^p$ to be $B^{k_1,k_2}$.
%For each $(k_1,k_2) \neq (0,0)$, we define $B^{{k_1,k_2}}$ based on two cases:
For each $(k_1, k_2)$, we define $B^{{k_1,k_2}}$ based on two cases:
\begin{itemize}
\item \underline{Case 1.} If $p_i(A^{k_1,k_2}) \geq \epsilon_p^{n+1}$, we put $B^{{k_1,k_2}} := A^{k_1,k_2}$.  The origin always 
falls into this case.
\item \underline{Case 2.} Otherwise, we solve the quadratic equation $p_i(rA^{k_1,k_2})=\epsilon_p^{n+1}$ for the unknown variable $r \in [0,1]$,
and define $B^{{k_1,k_2}} := r A^{k_1,k_2}$. 
\end{itemize}
After checking each vertex of $S^\rho_i$, we define
\begin{align}
R_i^{\rho,p} := \left[ 0, \Lambda_{-\hf,I_i} \right] \times \left[ 0, \Lambda_{+\hf,I_i} \right]
\subseteq S^p_i \subseteq S^\rho_i,
\end{align}
where
\begin{align}
\Lambda_{-\hf,I_i} := \min_{\substack{k_2 = 0, 1 }}\left( B^{1,k_2} \right), \quad
\Lambda_{+\hf,I_i} := \min_{\substack{k_1 = 0, 1 }}\left( B^{k_1,1} \right).
\end{align}

After performing this two step process at each grid cell $x_i$, the end result
of this construction is the following theorem.
\begin{theorem}
The numerical flux in Eqn. \eqref{eqn:flux} preserves positivity of the solution for any 
\begin{equation}
\label{eqn:theta-admissable-interval}
    \theta^n_{i-\hf} \in \left[0,\, \min\left( \Lambda_{+\hf,I_{i-1}},\Lambda_{-\hf, I_{i}}
\right) \right].
\end{equation}
\end{theorem}

Although we could in principle choose any value in the interval defined in
Eqn. \eqref{eqn:theta-admissable-interval} (e.g. $\theta^n_{i-\hf}=0$), in order
to retain high-order accuracy, we choose the largest possible value 
that we can prove retains positivity.  That is, we
define
\begin{align}
\label{eqn:theta}
    \theta^n_{i-\hf} := \min\left( \Lambda_{+\hf,I_{i-1}},\Lambda_{-\hf, I_{i}} \right)
\end{align}
at each cell interface.

This finishes the discussion for the 1D scheme.  We reiterate that
this entire process relies on flux modifications, which the Picard integral
formulation was designed to accept, and is pointed out in Rmk.  \ref{rmk:pif}.

\begin{rmk}
The positivity of the solution is guaranteed for the \underline{entire simulation}.
\end{rmk}

%   One consequence of being a single-stage, single-step method is that 
%   we do not have stages where the density or pressure can go negative.
%   In fact, when our method
%   is compared to previous finite difference or finite volume WENO schemes with
%   Runge-Kutta (RK) time discretizations \cite{XiQiXu14,ChLiTaXu14},
%   the authors relax the computational cost by applying the limiter only on the
%   final stage of the RK method, but in order to do so, 
%   they indicate they artificially define the sound speed as $c=\sqrt{\gamma |p|/|\rho|}$ 
%   in order to implement the characteristic decomposition for the high-order WENO reconstructions
%   for the intermediate stages.
%   Although it does not affect the refinement study in \cite{XiQiXu14,ChLiTaXu14}, 
%   this treatment may lead to a potential numerical instability for some extreme cases.
%   A similar issue can be found for the ideal MHD equations \cite{tang14}.

One consequence of being a single-stage, single-step method is that 
we do not have stages where the density or pressure can become negative,
whereas multistage Runge-Kutta methods typically introduce either additional computational cost
or artificial sound speeds in order to retain positivity.
For example, in \cite{ZhangShu10,ZhangShu11-survey,ZhangShu12} the limiter is applied after
each stage in the Runge-Kutta method.  This introduces additional computational
complexity 
(i.e.\ there are multiple applications of the limiter per time step)
as well as further constraints on the time step selection
because the limiter relies on positivity of the forward Euler method. The
modifications made in \cite{XiQiXu14,ChLiTaXu14} are part of an effort to decrease
the computational complexity by only applying the limiter once
per time step.  However, this happens at the expense of potentially introducing
negative pressure and density.   In order to compensate for this, in \cite{XiQiXu14,ChLiTaXu14} the authors
indicate they artificially define the sound speed as $c=\sqrt{\gamma
|p|/|\rho|}$ for each stage in the Runge-Kutta method.  This is necessary 
to implement the characteristic decomposition required for the
high-order WENO reconstruction, and although
it does not affect the refinement study in \cite{XiQiXu14,ChLiTaXu14}, 
this treatment may lead to a potential numerical instability for some extreme cases.
A similar issue can be found for the ideal MHD equations \cite{tang14}.

%   This is in
%   contrast to other recent work \cite{ZhangShu10,ZhangShu11-survey,ZhangShu12},
%   where the limiter is applied after each stage value.

%   the authors relax the computational cost by applying the limiter only on the
%   final stage of the RK method, but in order to do so, 
%   Although it does not affect the refinement study in \cite{XiQiXu14,ChLiTaXu14}, 
%   this treatment may lead to a potential numerical instability for some extreme cases.
%   A similar issue can be found for the ideal MHD equations \cite{tang14}.

\section{The positivity-preserving method: the 2D case} \label{sec:md}

In this section, we apply the positivity-preserving limiter to the two-dimensional case.
Extensions to a general multi-D case follow from what is provided here.

Recall that our single-stage, single-step update is given by
Eqn. \eqref{eqn:pif-2d}.  Similar to the 1D case,
we use $\hat{f}^n_{i-\hf,j}$ and $\hat{g}^n_{i,j-\hf}$ to denote the low-order positivity-preserving fluxes,
and our numerical method uses modified fluxes through
\begin{subequations}
\label{eqn:2dfluxes}
\begin{align}
\tilde{F}^n_{i-\hf,j} & := \theta^n_{i-\hf,j}(\hat{F}^n_{i-\hf,j} - \hat{f}^n_{i-\hf,j}) + \hat{f}^n_{i-\hf,j}, \\
\tilde{G}^n_{i,j-\hf} & := \theta^n_{i,j-\hf}(\hat{G}^n_{i,j-\hf} - \hat{g}^n_{i,j-\hf}) + \hat{g}^n_{i,j-\hf}.
\end{align}
\end{subequations}
Identical to the single-dimensional case, the positivity-preserving limiting procedure consists of two steps.
If we still use $\hat{\rho}^{n+1}$ and $\hat{p}^{n+1}$ to denote the low-order
density and pressure solved by the flux $\hat{f}^n$ and $\hat{g}^n$, we can
similarly define the 2D numerical lower bounds for density and pressure as
$\epsilon_\rho^{n+1}: = \min\left( \min_{i,j}\left( \hat{\rho}^{n+1}_{i,j}\right), \epsilon_0 \right)$
and 
$\epsilon_p^{n+1}: = \min\left( \min_{i,j}\left( \hat{p}^{n+1}_{i,j}\right), \epsilon_0 \right)$.

\subsection{Step 1: Maintain positivity of the density}

%First, we focus on the density $\rho^{n+1}_{i,j}$. 
Our fist step is to find local bounds 
$\Lambda^\rho_{L,I_{i,j}}$, $\Lambda^\rho_{R,I_{i,j}}$, $\Lambda^\rho_{U,I_{i,j}}$ and $\Lambda^\rho_{D,I_{i,j}}$,
such that for any $ (\theta^n_{i-\hf,j},\theta^n_{i+\hf,j},\theta^n_{i,j-\hf},\theta^n_{i,j+\hf}) \in S_{i,j}^\rho$,
we have
\begin{equation}
\label{inequ3}
	\rho^{n+1}_{i,j} = \rho^n_{i,j} 
	- \lambda_x \left( \tilde{f}^{n,\rho}_{i+\hf,j}-\tilde{f}^{n,\rho}_{i-\hf,j} \right) 
	- \lambda_y \left( \tilde{g}^{n,\rho}_{i,j+\hf}-\tilde{g}^{n,\rho}_{i,j-\hf} \right)\ge \epsilon^{n+1}_\rho,
\end{equation}
where
\begin{align}
 S_{i,j}^\rho :=
	\left[0, \Lambda^\rho_{L,I_{i,j}}\right] \times \left[0, \Lambda^\rho_{R,I_{i,j}}\right] \times
	\left[0, \Lambda^\rho_{D,I_{i,j}}\right] \times \left[0, \Lambda^\rho_{U,I_{i,j}}\right].
\end{align}
Again, we have used the notation $g^\rho$ to refer to the first component of the flux function, $g$.
We define the low-order update as
\begin{equation*}
	\hat{\rho}^{n+1}_{i,j} := {\rho}_{i,j}^n - \lambda_x \left(\hat{f}^{n,\rho}_{i+\hf,j}-\hat{f}^{n,\rho}_{i-\hf,j} \right) 
- \lambda_y \left(\hat{g}^{n,\rho}_{i,j+\hf}-\hat{g}^{n,\rho}_{i,j-\hf} \right),
\end{equation*}
and observe that it satisfies $\hat{\rho}^{n+1}_{i,j} \geq \epsilon^{n+1}_\rho > 0$ for all $i,j$, provided the density is positive at time $t^n$.
%% ------------------------------------------------------------------------- %%
%
%%%Similar to Eqn. \eqref{inequ2}, we rewrite Eqn. \eqref{inequ3} as
%%%\begin{align}
%%%\label{inequ3_1}
%%%\begin{split}
%%% & \lambda_x \theta^n_{i-\hf,j}(f^{n,\rho}_{i-\hf,j} - \hat{f}^{n,\rho}_{i-\hf,j}) 
%%%- \lambda_x \theta^n_{i+\hf,j}(f^{n,\rho}_{i+\hf,j} - \hat{f}^{n,\rho}_{i+\hf,j}) \\
%%%+ & \lambda_y \theta^n_{i,j-\hf}(g^{n,\rho}_{i,j-\hf} - \hat{f}^{n,\rho}_{i,j-\hf}) 
%%%- \lambda_y \theta^n_{i,j+\hf}(g^{n,\rho}_{i,j+\hf} - \hat{f}^{n,\rho}_{i,j+\hf}) 
%%%   \geq  - \Gamma_{i,j}.
%%%   \end{split}
%%%\end{align}
%%%For abbreviation, we define 
%%%\begin{align}
%%%\begin{cases}
%%%\Delta f_{i-\hf,j} := \lambda_x (f^{n,\rho}_{i-\hf,j} - \hat{f}^{n,\rho}_{i-\hf,j}), \\
%%%\Delta f_{i+\hf,j} := - \lambda_x (f^{n,\rho}_{i+\hf,j} - \hat{f}^{n,\rho}_{i+\hf,j}), \\ 
%%%\Delta f_{i,j-\hf} := \lambda_y (g^{n,\rho}_{{i,j-\hf}} - \hat{g}^{n,\rho}_{i,j-\hf}), \\
%%%\Delta f_{i,j+\hf} := - \lambda_y (g^{n,\rho}_{i,j-\hf} - \hat{g}^{n,\rho}_{i,j-\hf}). \end{cases}
%%%\end{align}
%%%This reduces Eqn. $\eqref{inequ3_1}$ to
%%%\begin{equation}
%%%\label{inequ4}
%%%\theta^n_{i+\hf,j} \Delta f_{i+\hf,j} + \theta^n_{i-\hf,j} \Delta f_{i-\hf,j}+
%%%\theta^n_{i,j+\hf} \Delta f_{i,j+\hf} + \theta^n_{i,j-\hf} \Delta f_{i,j-\hf} \ge - \Gamma_{i,j}.
%%%\end{equation}
Similar to Eqn. \eqref{inequ2}, we rewrite Eqn. \eqref{inequ3} as
\begin{equation}
\label{inequ4}
 \theta^n_{i-\hf,j} \Delta f_{i-\hf,j} - \theta^n_{i+\hf,j} \Delta f_{i+\hf,j} +
 \theta^n_{i,j-\hf} \Delta g_{i,j-\hf} - \theta^n_{i,j+\hf} \Delta g_{i,j+\hf} \ge \epsilon_\rho^{n+1} - \hat{\rho}^{n+1}_{i,j},
\end{equation}
where we have defined the deviation between the high- and low-order fluxes as
\begin{align}
\begin{cases}
	\Delta f_{i-\hf,j}   := \lambda_x (\hat{F}^{n,\rho}_{i-\hf,j}    - \hat{f}^{n,\rho}_{i-\hf,j}), \\
	\Delta f_{i+\hf,j}  := \lambda_x (\hat{F}^{n,\rho}_{i+\hf,j}   - \hat{f}^{n,\rho}_{i+\hf,j}), \\ 
	\Delta g_{i,j-\hf}  := \lambda_y (\hat{G}^{n,\rho}_{{i,j-\hf}} - \hat{g}^{n,\rho}_{i,j-\hf}), \\
	\Delta g_{i,j+\hf} := \lambda_y (\hat{G}^{n,\rho}_{i,j-\hf}   - \hat{g}^{n,\rho}_{i,j-\hf}). 
\end{cases}
\end{align}
%\todo{Unless we state: ``at grid cell i, we define ...'', then these definitions do not make sense.
%That is, if I'm at cell $x_i$, $f_{i-\hf,j}$ is not the same as $f_{i+\hf,j}$ from grid cell $x_{i-1}$.}

Similar to the 1D case, we solve $\eqref{inequ4}$ based on the signs of $\Delta f_{i\pm \hf,j}$ and $\Delta g_{i,j\pm\hf}$
at each node $(x_i, y_j)$.  The basic idea requires a total of two steps:
\begin{enumerate}
  \item Identify the negative values of each of the four numbers
\begin{equation}
		\left\{ \Delta f_{i-\hf,j},\, -\Delta f_{i+\hf,j},\, \Delta g_{i,j-\hf},\, -\Delta g_{i,j+\hf} \right\}.
\end{equation}
  \item Corresponding to the collective negative values, we define upper
  bounds on the limiting parameters by solving Eqn. \eqref{inequ4} for each value of $\theta$ after neglecting any positive 
  values found.
  For example, if $\Delta f_{i-\hf,j}, -\Delta f_{i+\hf,j} < 0$ 
  and $\Delta g_{i,j-\hf}, -\Delta g_{i,j+\hf} \ge 0$, then we define
 \begin{align}
 \begin{cases}
 \Lambda^\rho_{L,I_{i,j}} := \Lambda^\rho_{R,I_{i,j}} := \min \left( \frac{\epsilon_\rho^{n+1} - \hat{\rho}^{n+1}_{i,j}}
 { \Delta f_{i-\hf,j}-\Delta f_{i+\hf,j}} , 1\right), \\
    \Lambda^\rho_{D,I_{i,j}} := \Lambda^\rho_{U,I_{i,j}} := 1.
    \end{cases}
    \end{align}
    Likewise, if $-\Delta f_{i+\hf,j}, \Delta g_{i,j-\hf} < 0$ 
  and $\Delta f_{i-\hf,j}, -\Delta g_{i,j+\hf} \ge 0$, then we define
 \begin{align}
 \begin{cases}
 \Lambda^\rho_{R,I_{i,j}} := \Lambda^\rho_{D,I_{i,j}} := \min \left( \frac{\epsilon_\rho^{n+1} - \hat{\rho}^{n+1}_{i,j}}
 { -\Delta f_{i+\hf,j}+\Delta g_{i,j-\hf}} , 1\right), \\
    \Lambda^\rho_{L,I_{i,j}} := \Lambda^\rho_{U,I_{i,j}} := 1.
    \end{cases}
    \end{align}
    There are a total of 16 cases.  Each follow similarly, and are omitted for brevity.
%\todo{How many cases are there to consider? \textbf{16}. I see the four 1D cases for each of 
%$\Delta f_{i+\hf,j}, \Delta f_{i-\hf,j} \geq 0$ or
%$\Delta g_{i,j+\hf}, \Delta g_{i,j-\hf} \geq 0$, but what about the ``cross''-terms? \textbf{The formula is same.}
%Basically, it just needs to take care of the negative terms.
%}
  \end{enumerate}

\subsection{Step 2: Maintain positivity of the pressure}

Using the same construction from \S\ref{subsec:step2-1d}, we 
%only need to 
%identify a convex subset of the set $S_{i,j}^\rho$
identify a rectangle $R_{i,j}^{\rho,p} \subseteq S_{i,j}^\rho$
where the pressure satisfies $p_{i,j}(\theta^n_{i-\hf,j},\theta^n_{i+\hf,j},\theta^n_{i,j-\hf},\theta^n_{i,j+\hf}) \geq \epsilon_p^{n+1} $.
%This subset $S_{i,j}^p$ can identified as follows.
%
Again, we consider the vertices of the region that were computed in the first step.  In 2D, we define
them as
\[
A^{k_1,k_2,k_3,k_4} := (k_1 \Lambda^\rho_{L,I_{ij}}, k_2 \Lambda^\rho_{R,I_{ij}},
k_3 \Lambda^\rho_{D,I_{ij}}, k_4 \Lambda^\rho_{U,I_{ij}}),
\quad
k_1, k_2, k_3, k_4 \in \{0,1 \}.
\]

We rescale each vertex in an identical manner to the 1D case presented in 
subsection \ref{subsec:step2-1d}.  There are two cases:
\begin{itemize}
\item \underline{Case 1.} If $p_{i,j}(A^{k_1,k_2,k_3,k_4}) \geq \epsilon_p^{n+1}$, we 
define the vertex $B^{{k_1,k_2,k_3,k_4}} := A^{k_1,k_2,k_3,k_4}$.
\item \underline{Case 2.} 
We solve the quadratic equation $p_{i,j}(rA^{k_1,k_2,k_3,k_4}) = \epsilon_p^{n+1}$ for the unknown
$r \in [0,1]$ and put $B^{{k_1,k_2,k_3,k_4}} := r A^{k_1,k_2,k_3,k_4}$.
%
%Otherwise, we solve the quadratic equation $p_i(rA^{k_1,k_2})=10^{-13}$ for the unknown variable $r \in [0,1]$,
%and define $B^{{k_1,k_2}} := r A^{k_1,k_2}$. 
\end{itemize}
%
%After rescaling, if necessary, we define $S_{i,j}^{p}$ as the
%the convex four-dimensional polyhedra defined by 
%the fifteen vertices $B^{k_1,k_2,k_3,k_4}$ together with the origin $(0,0,0,0)$.
%
In the final step, we identify a rectangular box inside $S_{i,j}^p$ through
\begin{align}
R_{i,j}^{\rho,p} := [0, \Lambda_{L,I_{i,j}}] \times [0,\Lambda_{R,I_{i,j}}] \times [0, \Lambda_{D,I_{i,j}}] \times [0,\Lambda_{U,I_{i,j}}],
\end{align}
where
\begin{equation}
\begin{aligned}
\Lambda_{L,I_{i,j}} := \min_{{k_{2,3,4} \in \{0, 1\}}}\left( B^{1,k_2,k_3,k_4} \right), \quad
\Lambda_{R,I_{i,j}} := \min_{{k_{1,3,4} \in \{0, 1\}}}\left( B^{k_1,1,k_3,k_4} \right), \\
\Lambda_{D,I_{i,j}} := \min_{{k_{1,2,4} \in \{0, 1\}}}\left( B^{k_1,k_2,1,k_4} \right), \quad
\Lambda_{U,I_{i,j}} := \min_{{k_{1,2,3} \in \{0, 1\}}}\left( B^{k_1,k_2,k_3,1} \right).
\end{aligned}
\end{equation}
After repeating this procedure for each node $(x_i, y_j)$, we finish by defining the
scaling parameter as
%\begin{subequations}
%\begin{align}
%\theta^n_{i-\hf,j} := \min(\Lambda_{R,I_{i-1,j}},\Lambda_{L, I_{i,j}}), \\
%\theta^n_{i,j-\hf} := \min(\Lambda_{U,I_{i,j-1}},\Lambda_{D, I_{i,j}}),
%\end{align}
%\end{subequations}
\begin{equation}
\theta^n_{i-\hf,j} := \min(\Lambda_{R,I_{i-1,j}},\Lambda_{L, I_{i,j}}), \quad
\theta^n_{i,j-\hf} := \min(\Lambda_{U,I_{i,j-1}},\Lambda_{D, I_{i,j}}),
\end{equation}
and insert the result into Eqn \eqref{eqn:2dfluxes} to define our modified fluxes.
This finishes the discussion for the 2D scheme, and a similar
positivity-preserving theorem follows as in Thm.
\ref{eqn:theta-admissable-interval}.

\section{Numerical results} \label{sec:numerical-results}

In this section, we perform numerical simulations with our proposed positivity-preserving method on 1D and 2D compressible Euler equations. 

\subsubsection{Implementation details}

The parameters we use for our WENO reconstructions include a power parameter $p=2$, and a regularization parameter 
$\eps = 10^{-6}$, and a  gas constant of $\gamma = 1.4$.  In addition, we follow common practice and use a global (as opposed to a local) value for $\alpha$ in the Lax-Friedrichs
flux splitting for all of our simulations.  This introduces additional
numerical dissipation that helps to prevent unphysical oscillations in this high-order scheme.  Contrary to what typically happens
with first-order finite volume schemes, the additional numerical dissipation introduced by this choice does not introduce an exorbitant amount of artificial diffusion.
In every simulation save one, the CFL number is $0.35$.
All of our numerical results can be found in the open source software package FINESS \cite{FINESS}.

\subsection{Accuracy test}
To test the accuracy of our method,
we use the smooth vortex problem with low density and low pressure \cite{XiQiXu14,ZhangShu12}.
Initially, we have a mean flow
\begin{align}
(\rho, u^1, u^2, u^3, p) = (1, 1, 1, 0, 1),
\end{align}
with perturbations on the velocities $u^1$, $u^2$ and the temperature $T = {p}/{\rho}$, given by
%\begin{gather*}
%(\delta u^1, \delta u^2) = \frac{\epsilon}{2\pi}e^{0.5(1-r^2)}(-{y},{x}), \\
%\delta T = - \frac{(\gamma-1)\epsilon^2}{8\gamma\pi^2}e^{1-r^2}.
%\end{gather*}
\begin{equation*}
(\delta u^1, \delta u^2) = \frac{\epsilon}{2\pi}e^{0.5(1-r^2)}(-{y},{x}), \quad
\delta T = - \frac{(\gamma-1)\epsilon^2}{8\gamma\pi^2}e^{1-r^2}, \quad
r^2 := {x}^2 + {y}^2.
\end{equation*}
%Here $r^2={x}^2 + {y}^2$. 
%
The initial pressure and density are determined by keeping the entropy $S = p/{\rho^\gamma}$ constant.
The domain is $(x,y) \in [-5,5]\times[-5,5]$ with 
periodic boundary condition on all sides.
The vortex strength $\epsilon$ is set as $10.0828$ such that the lowest density and lowest 
pressure in the center of the vortex are $7.8 \times 10^{-15}$ and $1.7 \times 10^{-20}$ respectively.

A convergence study is presented in Table \ref{tab:2dhd}.  The $L_1$-errors
and $L_\infty$-errors of the density are computed at a final time of $t = 0.01$.
We observe the fifth-order accuracy of the proposed scheme, which is 
comparable with those demonstrated in \cite{XiQiXu14,ZhangShu12}.
In \cite{ZhangShu12}, the authors took $\Delta t = \Delta x^{\frac{5}{3}}$ in
order to make the spatial error dominate the numerical error. 
We find this treatment is not necessary to observe high-order spatial accuracy
because of the short the final time.  In our table, we only present the
results with a constant CFL number of $0.35$
that has been chosen for this, and all other simulations save one.
Without the addition of the positivity limiter,
we observe negative density and negative pressure with the Taylor formulation of the PIF-WENO scheme that
appears in the center of vortex.

\begin{table}
\begin{center}
\begin{Large}
    \caption{Accuracy test of the 2D smooth vortex. We show the $L_{1}$-errors and $L_{\infty}$-errors at time $t = 0.01$ of the density. 
    The solutions converge at fifth-order accuracy.   \label{tab:2dhd}}
    \end{Large}
\begin{tabular}{|c || c c c c |}
  \hline
{\normalsize{ Mesh}} & {\normalsize  $L_{1}$-Error}  &  {\normalsize  Order} & {\normalsize $L_{\infty}$-Error } & {\normalsize Order } \\ \hline \hline
{\normalsize $80 \times 80$}  	&2.970e-06	&	-	&	2.494e-04& 	-	\\
{\normalsize $160 \times 160$}  &1.627e-07	&	4.190	&	2.442e-05&	3.353	\\
{\normalsize $320 \times 320$} 	& 7.384e-09	&	4.462	&	1.390e-06&	4.135	\\
{\normalsize $640 \times 640$} &2.428e-10	&	 4.927	&	4.718e-08&	4.881	\\ \hline
 \end{tabular}
\end{center}
\end{table}

\subsection{1D Sedov blast wave problem}
The first 1D problem we considered is the 1D Sedov blast problem originally from the book by Sedov \cite{se59}.
The problem describes an intense explosion in a gas where the disturbed air is separated from the undisturbed air by a shock wave.
Initially, we deposit a quantity of energy $\E=3200000$ into the center cell of the domain with the length of $\Dx$,
and the energy in every other cell is set to $10^{-12}$.
The other quantities are initialized with a constant values of $\rho = 1$ and $u^1 = 0$.
The numerical results are displayed in Fig.~\ref{1dsed}, where we see the shock wave is captured with the proposed limiter used.
In Fig.~\ref{1dsed}, we use the exact solution given in Sedov's book \cite{se59} as the reference solution.
Our results are in agreement with other recent work \cite{ZhangShu10,ZhangShu12,XiQiXu14}.

\subsection{1D double rarefaction problem}
The second 1D problem we considered is the double rarefaction problem. It is a
Riemann problem with an initial condition of $\rho_L = \rho_R = 7$, $u^1_L =
-1$, $u^1_R = 1$ and $p_L = p_R = 0.2$.  The exact solution consists of two
rarefaction waves traveling in opposite directions, which results in the
creation of a vacuum in the center of the domain.  Only with the proposed
limiter are we are able to solve this low density and low pressure problem
with the high-order finite difference WENO method.  For this problem only, we find it is
necessary to reduce the CFL number from $0.35$ to $0.15$ in order to avoid
introducing spurious oscillations near the top of the rarefactions.  The
numerical results are presented in Fig.~\ref{1dst}, where we use the same
resolution of $\Dx = 1/200$ as those in reference
\cite{ZhangShu10,ZhangShu12,XiQiXu14}.  Our
results are in agreement with the exact solution.  Here, the exact
solution is a highly resolved solution with $\Dx = 1/1000$.  Other Riemann
problems have been investigated, and our method gives similar results as those
found elsewhere in the literature (e.g.  \cite{EsVi96,TangXu99}).

\subsection{2D Sedov blast wave problem}
We also consider a 2D version of the Sedov blast wave problem.
In 2D case, the problem has an exact self-similar solution and we expect the numerical result has a similar structure.
In the simulation, we only compute one quadrant of the whole domain, where we
choose the computation domain to be $(x,y) \in [0, 1.1] \times [0, 1.1]$.
Similar to the 1D case, we deposit a quantity of energy $\E=0.244816$ into the lower left corner cell, and
set the energy in every other cell to be $\E=10^{-12}$.
The other initial values are identical to the 1D case.
We apply solid wall boundary conditions along the bottom ($x=0$) and left
($y=0$) boundaries so that 
the setup is equivalent to computing the whole domain $[-1.1, 1.1] \times [-1.1, 1.1]$ with $\E=0.979264$.
The density is presented in Fig.~\ref{2dsed},
from which we can see the result has a nice self-similar structure.
Additionally, we observe that the density cut at $y=0$ agrees with the exact solution.
%Again, our results are in agreement with other recent works \cite{ZhangShu10,ZhangShu12,XiQiXu14}.

\subsection{2D Shock diffraction problem}
The second 2D problem we consider is the shock diffraction problem. 
The computational domain is $  [0, 1]\times [6, 11]   \bigcup [1, 13]\times [0, 11]$.
There is a shock wave of Mach number $5.09$ initially located at $\{x = 0.5, 6 \leq y \leq 11\}$.
As time evolves, the wave moves into undisturbed air with a density of $\rho = 1.4$ and pressure of $p = 1$. 
We use an inflow boundary condition at $\{x = 0, 6 \leq y \leq 11\}$, and an outflow boundary condition at $\{x = 13, 0 \leq y \leq11 \} $, $ \{ 1 \leq x \leq 13, y = 0 \}$ and $ \{0 \leq x \leq 13, y = 11 \}$.
For the other parts of the boundary where $\{ 0 \leq x \leq 1, y = 6 \}$ and $ \{ x = 1,  0 \leq y \leq 6 \} $,
solid wall boundary conditions are applied. 
As the shock passes the corner, negative density and negative pressure is observed without the addition of the
positivity limiter to the Taylor PIF-WENO scheme.
This issue is resolved with the proposed modifications to the older scheme. 
In Fig.~\ref{2dsd}, we present results for the density and pressure at time $t = 2.3$.
%The results are consistent with those solutions solved by different schemes in \cite{ZhangShu10,ZhangShu12,XiQiXu14,ChLiTaXu14}.

\begin{figure}
\begin{center}
\begin{tabular}{cc}
(a)\includegraphics[width=0.42\textwidth]{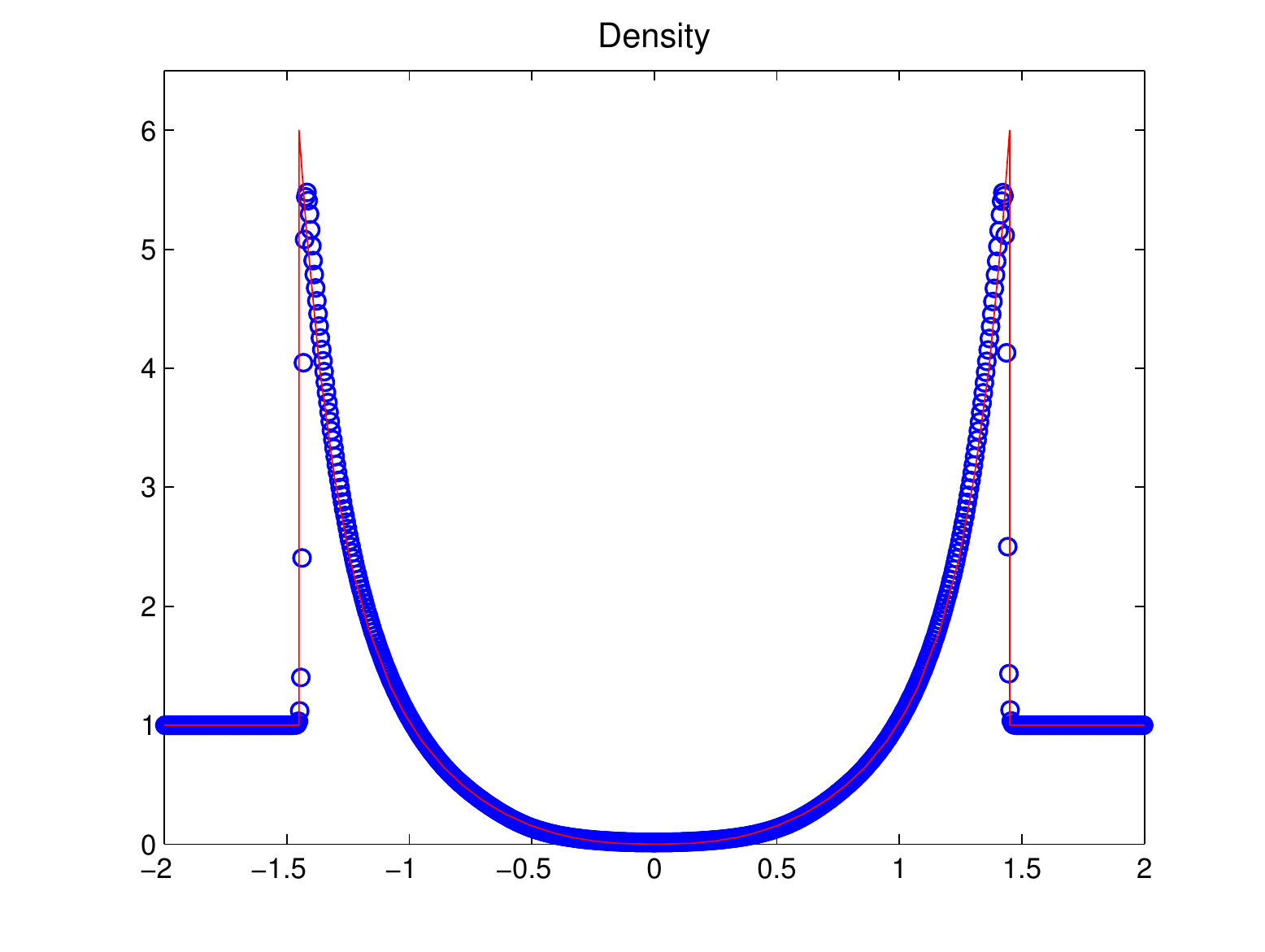} &
(b)\includegraphics[width=0.42\textwidth]{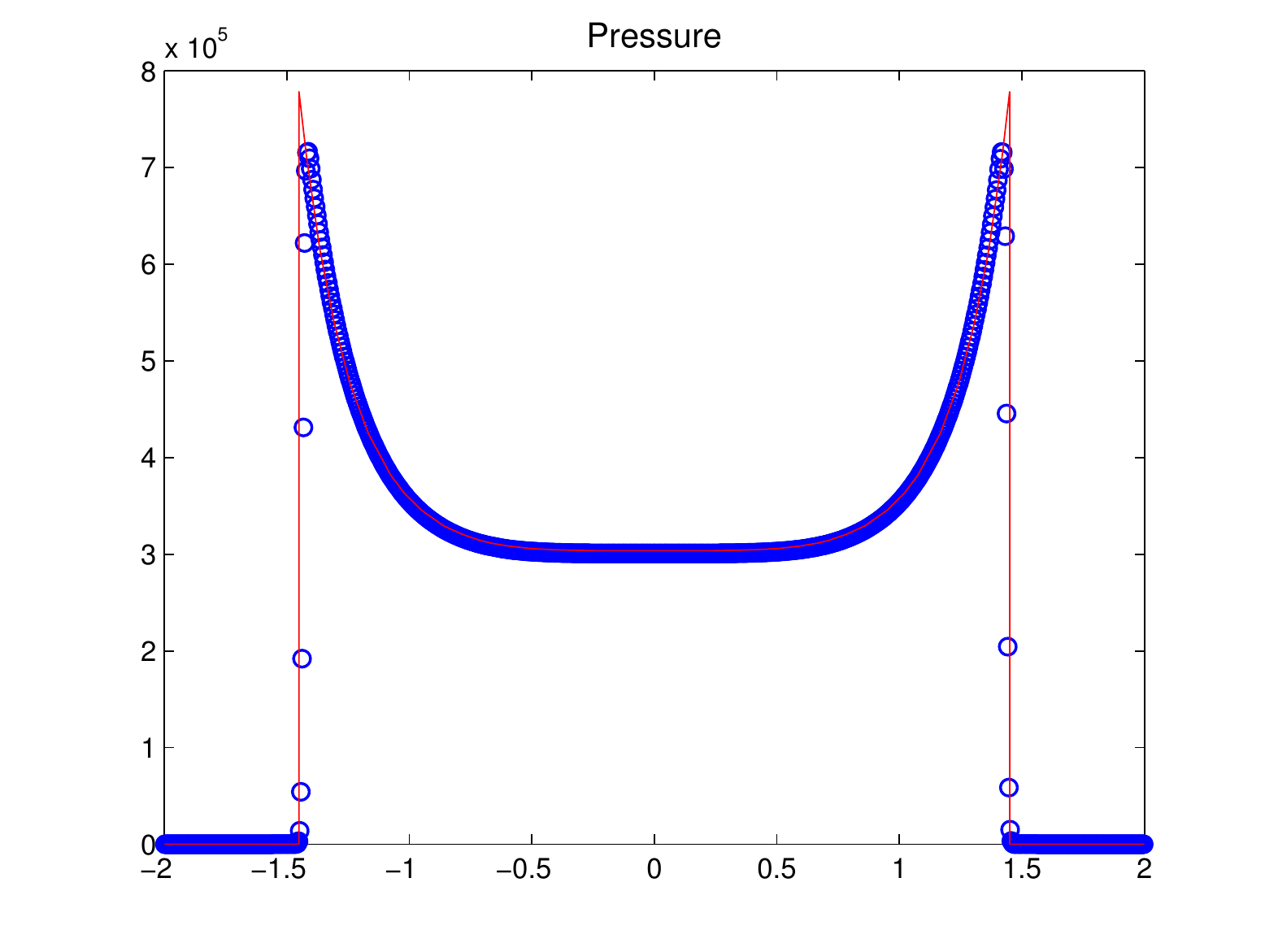} \\
\end{tabular}
(c)\includegraphics[width=0.42\textwidth]{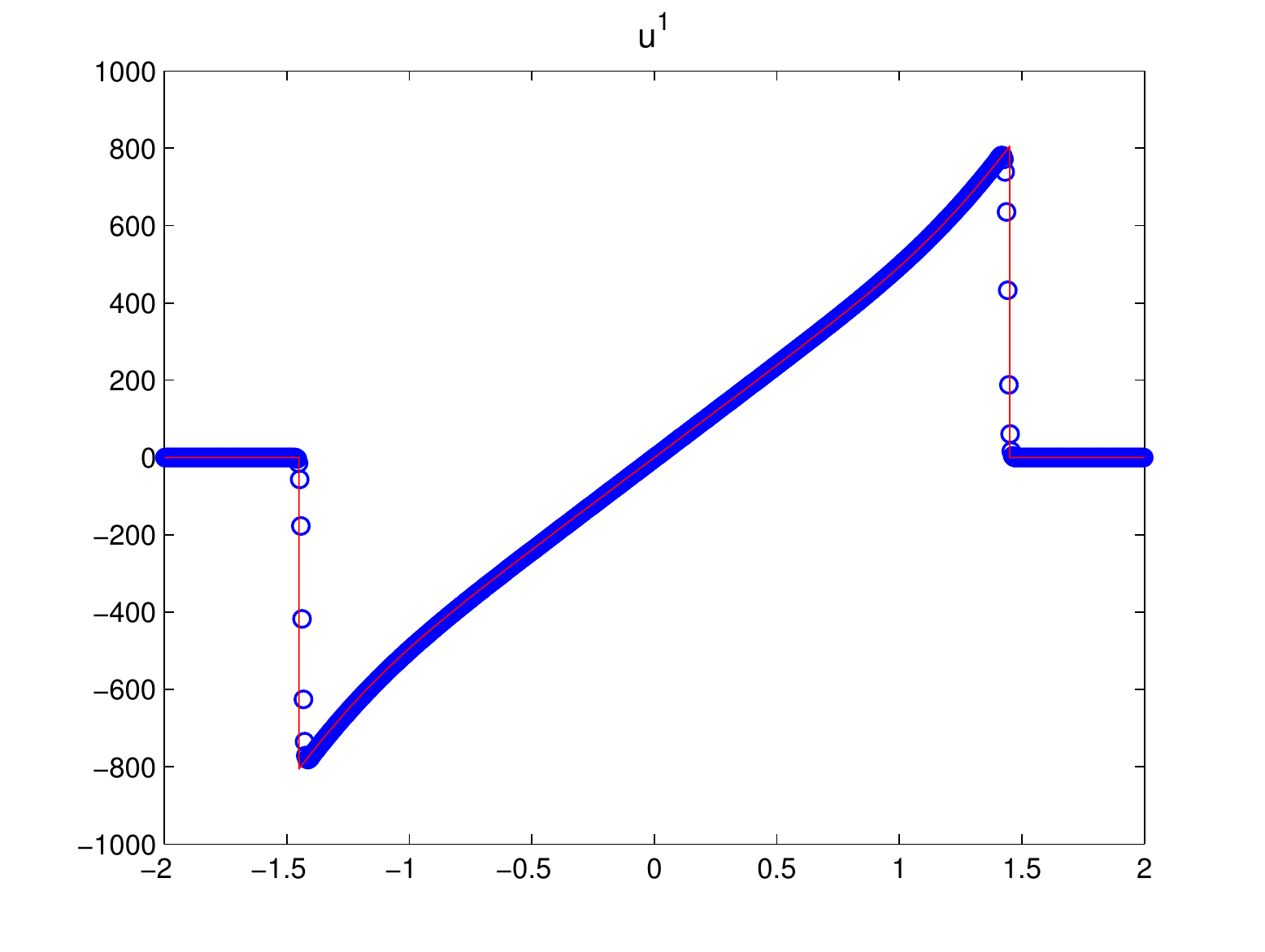}
      \caption{1D Sedov blast wave problem.
 These panels show plots at time $t = 0.001$ of (a) the density, (b) the
 pressure and  (c) the velocity $u^1$. The solid lines are the exact solutions. The
 solution was obtained on a mesh with $\Dx = 1/200$ and a CFL of 0.35.  
\label{1dsed}}
\end{center}
\end{figure}

\begin{figure}
\begin{center}
\begin{tabular}{cc}
(a)\includegraphics[width=0.42\textwidth]{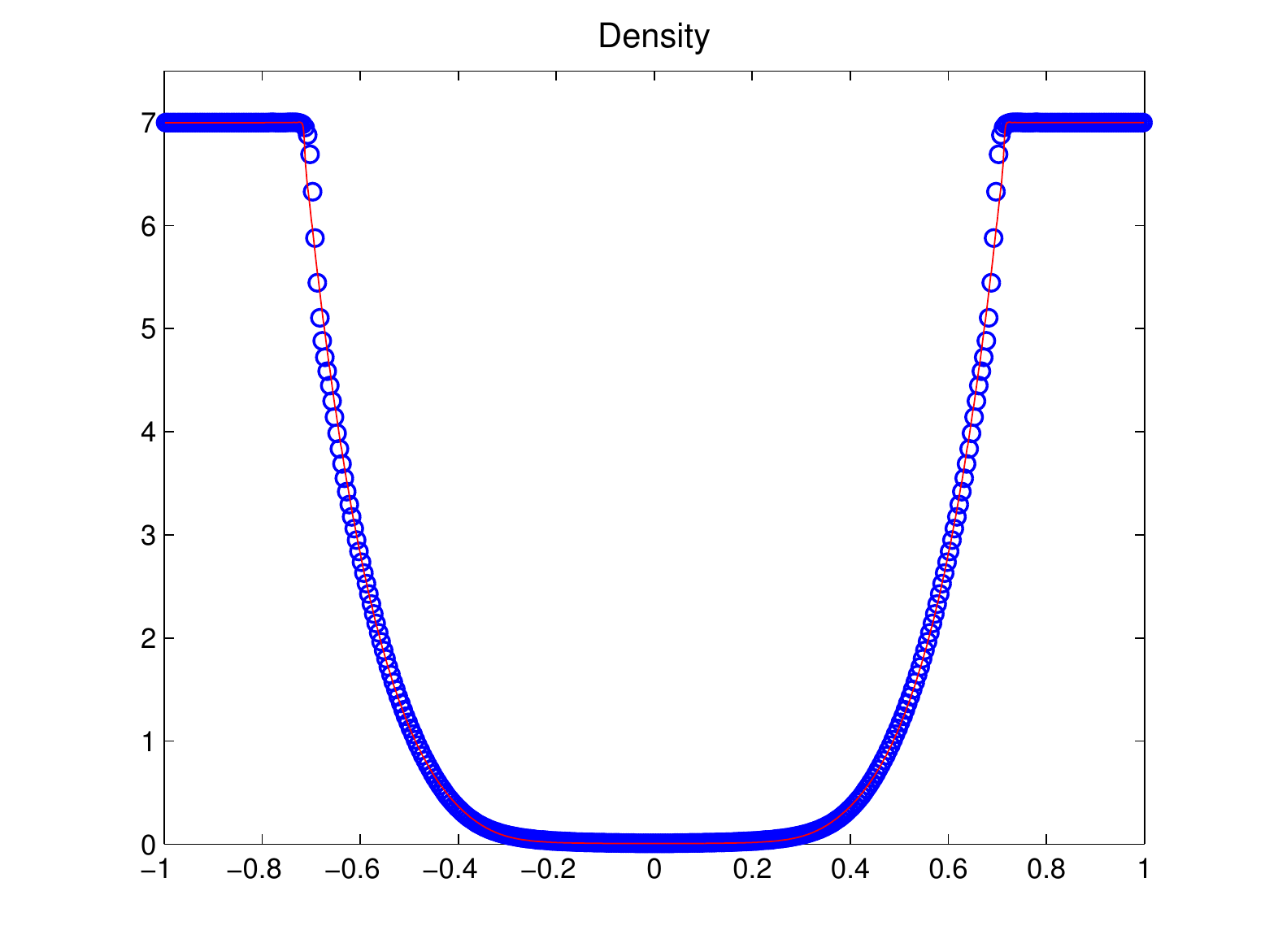} &
(b)\includegraphics[width=0.42\textwidth]{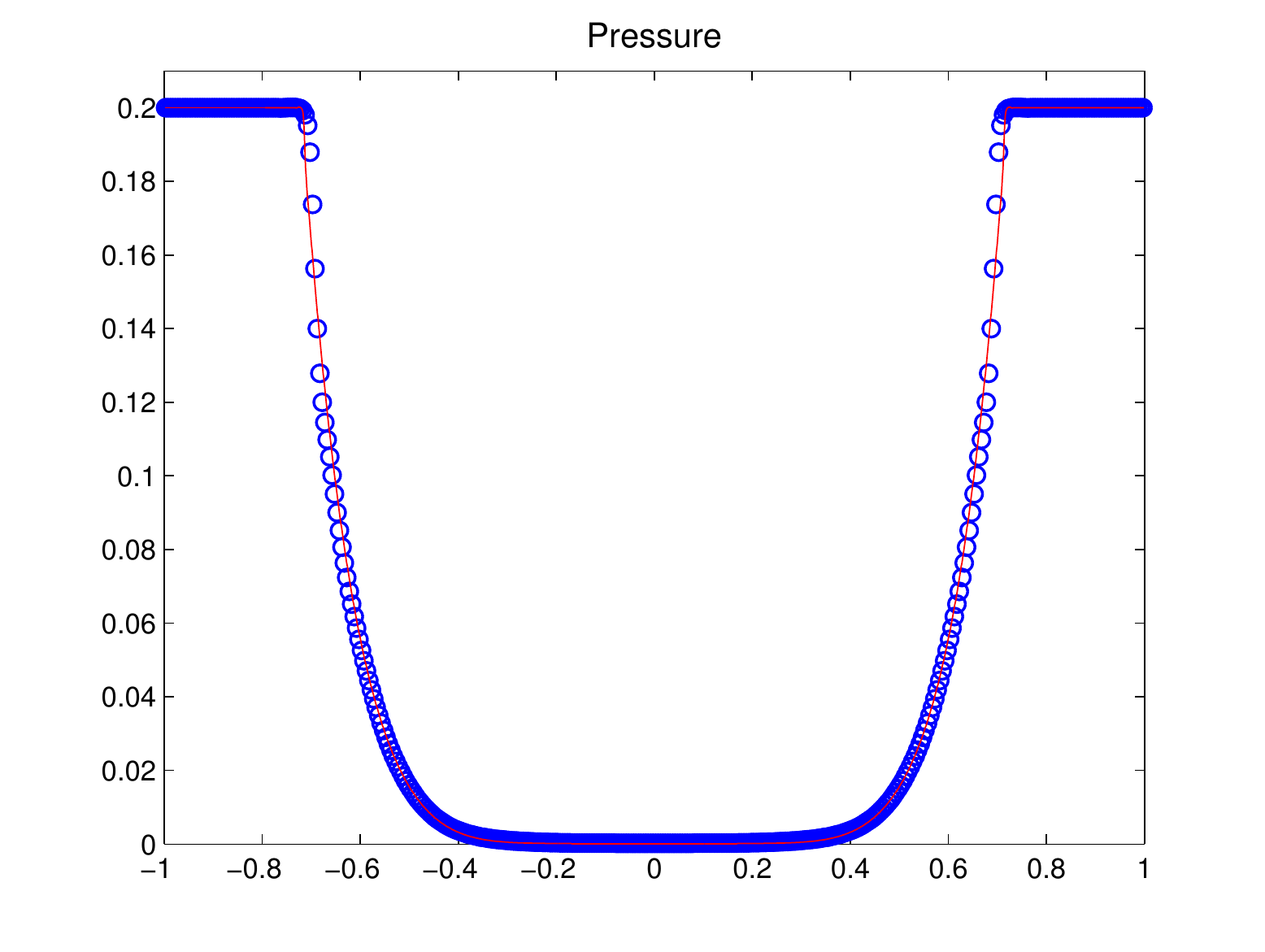} \\
\end{tabular}
(c)\includegraphics[width=0.42\textwidth]{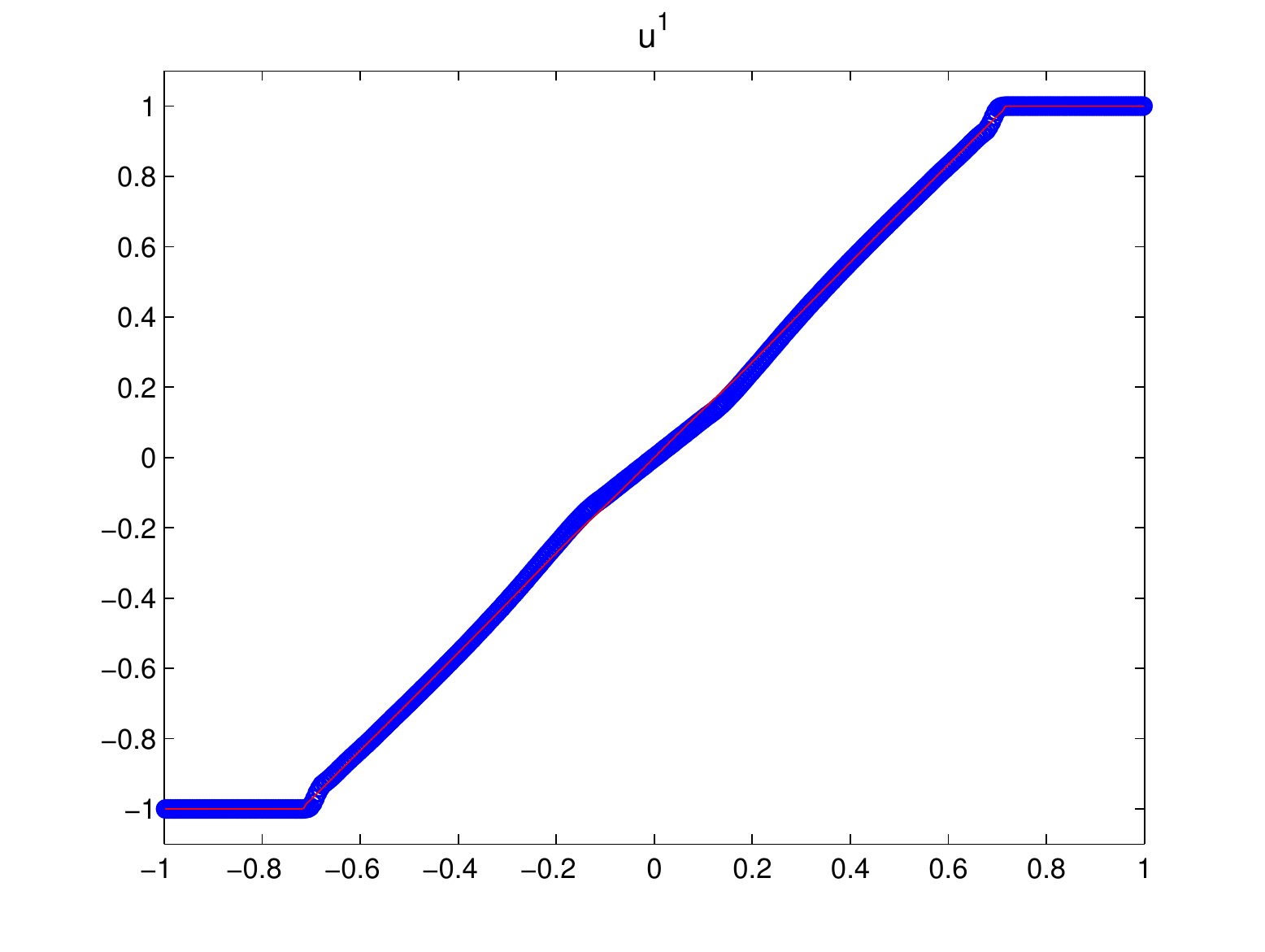}
      \caption{1D double rarefaction problem.
These panels show plots at time $t = 0.6$ of (a) the density, (b) the
pressure and  (c) $u^1$. The solid lines are the exact solutions. The solution
is obtained on a mesh with $\Dx = 1/200$ and a smaller CFL of $0.15$ that
help to reduce unphysical oscillations. \label{1dst}}
\end{center}
\end{figure}

\begin{figure}
\begin{center}
\begin{tabular}{cc}
(a)\includegraphics[width=0.42\textwidth]{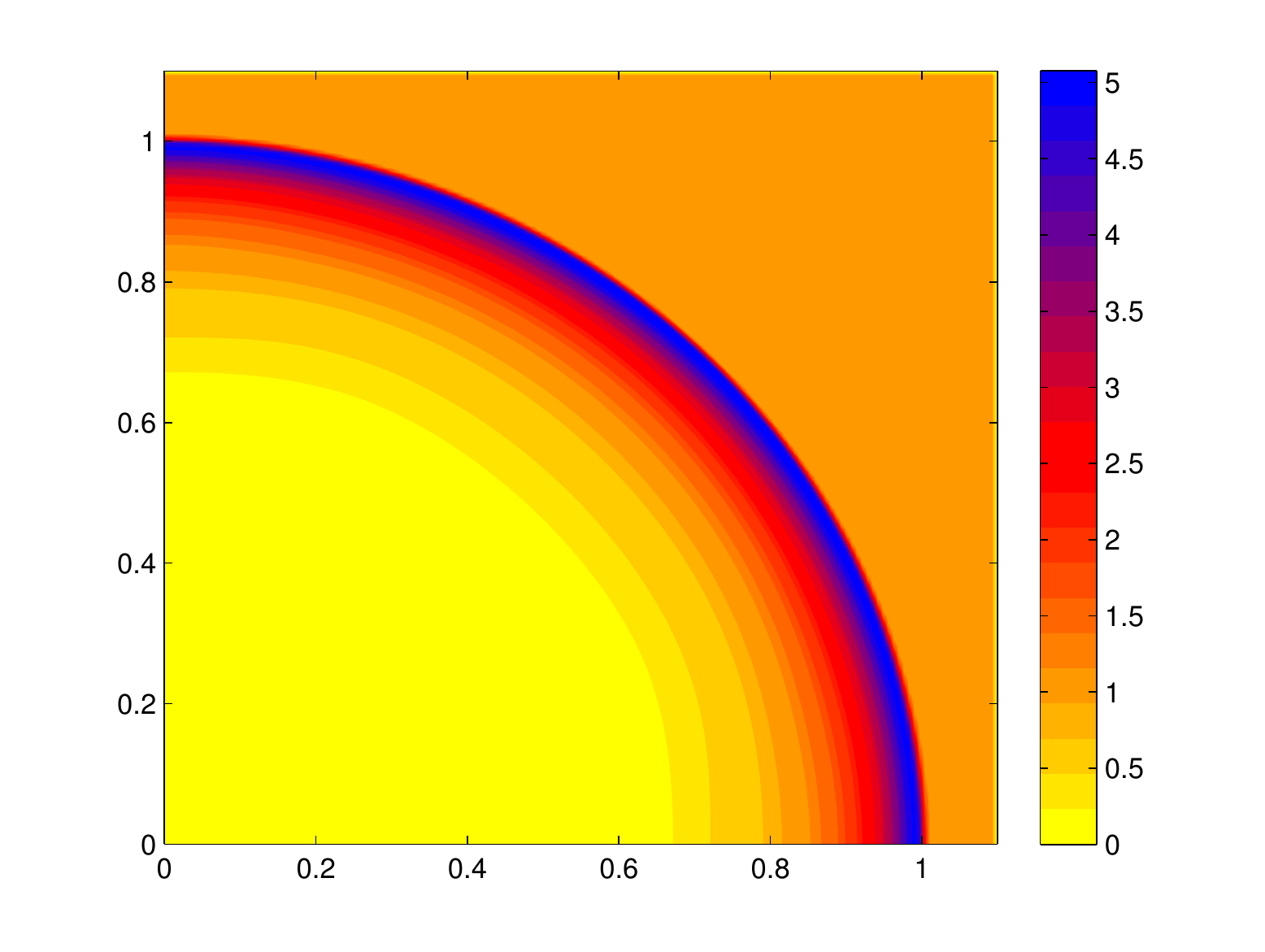} &
(b)\includegraphics[width=0.42\textwidth]{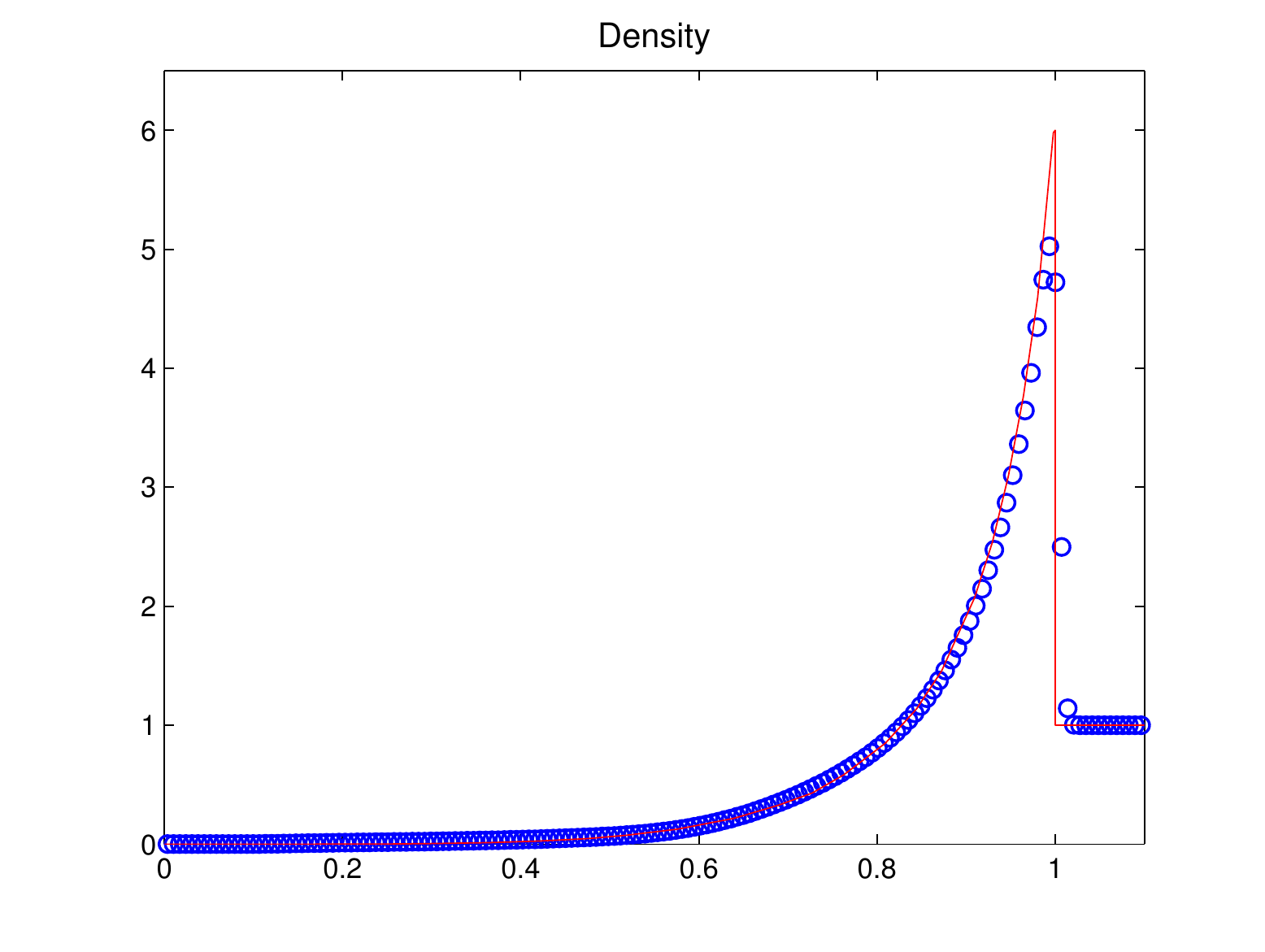} \\
\end{tabular}
      \caption{2D Sedov blast wave problem.
These panels show plots at time $t = 1$ of (a) the density, and (b) a
slice of the density along $y = 0$. The solid line in (b) is the exact solution. The
solution is obtained on a $160 \times 160$ mesh and a CFL number of $0.35$.  \label{2dsed}}
\end{center}
\end{figure}

\begin{figure}
\begin{center}
\begin{tabular}{cc}
(a)\includegraphics[width=0.42\textwidth]{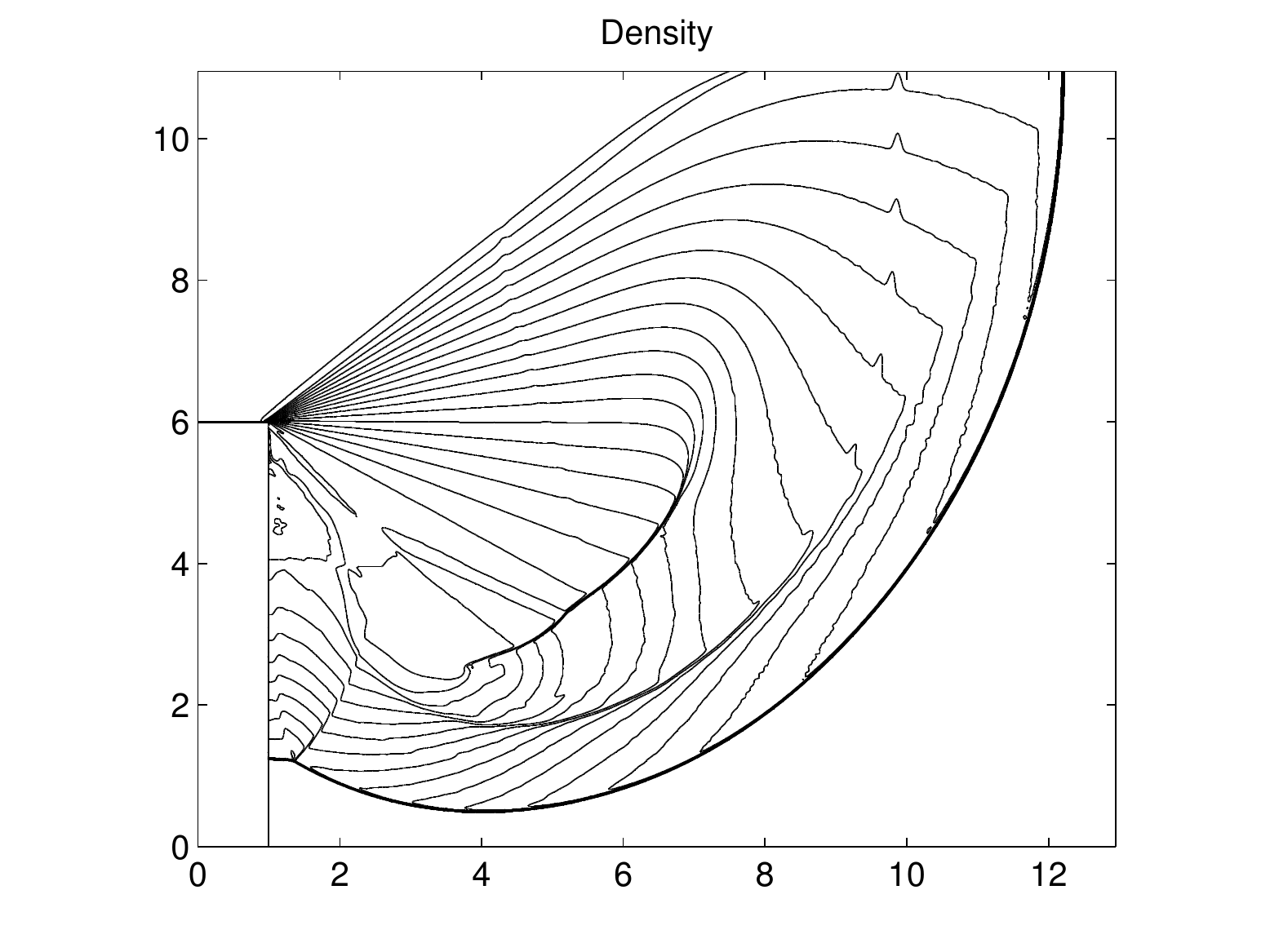} &
(b)\includegraphics[width=0.42\textwidth]{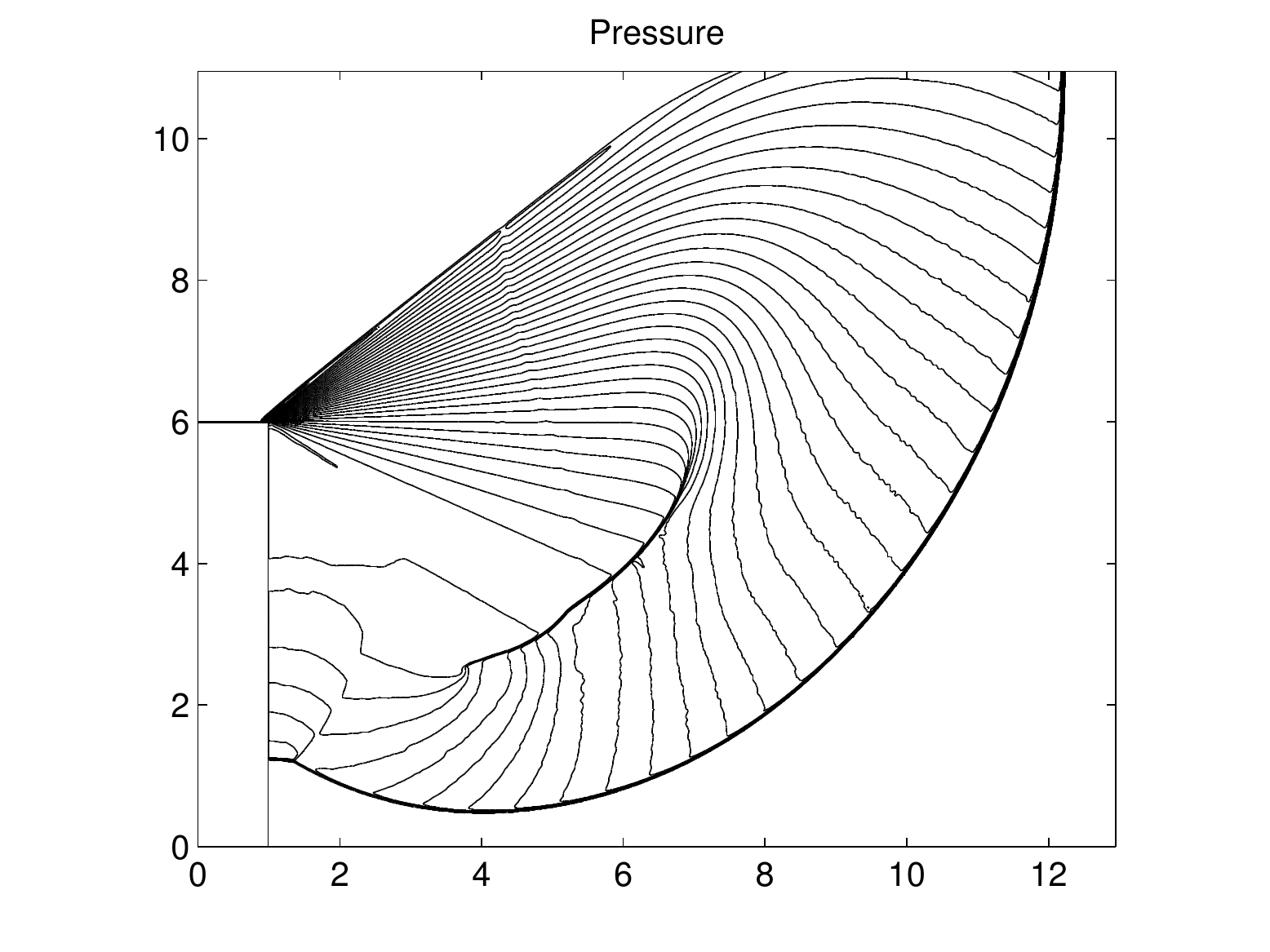} \\
\end{tabular}
      \caption{2D Shock diffraction problem. These panels show plots
      at time $t = 2.3$ of (a) the density, and (b) the pressure. A total of
      20 equally spaced contour lines from $\rho = 0.0662$ to $7.07$ are
      plotted in (a). We use 40 equally spaced contour lines from $p = 0.091$
      to $37$ are in (b). The solution is computed on a mesh with $\Dx =\Dy =
      1/30$ and a CFL number of $0.35$. \label{2dsd}
}
\end{center}
\end{figure}
%% ------------------------------------------------------------------------- %%
%% Section: Conclusions
%% ------------------------------------------------------------------------- %%
\section{Conclusions and future work} \label{sec:conclusions}

In this paper we propose a high-order, single-stage, single-step, positivity-preserving 
method for the compressible Euler equations. 
The base scheme is the Taylor discretization of the Picard integral formulation of the PDE, where
a single finite difference WENO reconstruction is applied once per time step.
A positivity-preserving limiter is introduced in such a way that the positivity of the solution 
is preserved for the \emph{entire simulation}, which adds a degree of robustness to our scheme.
In addition, we have no excessive CFL restriction in order to retain
positivity, which makes our new method more efficient compared to recent positivity-preserving methods.
We demonstrate the effectiveness and efficiency of the positivity-preserving 
schemes on one- and two-dimensional problems with low density and pressure.
High-order accuracy is retained after the introduction of our positivity
preserving limiter on a test problem that has near zero pressure and density.
Future work includes the construction of positivity-preserving multiderivative
methods \cite{Seal2014}, applying the proposed method to other systems such as 
magnetohydrodynamics, as well as incorporating our method into an AMR framework. 

\noindent
\begin{acknowledgements}
The authors would like to thank the anonymous reviewer for the helpful suggestions to further improve this work.
%\todo{what grants do we need to reference?}
This work has been supported in part by 
Air Force Office of Scientific Research %AFOSR %Air Force Office of Scientific Research 
grants 
FA9550-11-1-0281,  % Implicit maxwell
FA9550-12-1-0343   % Ultra-cold plasma
and
FA9550-12-1-0455,  % Ben Ong's "Fault Tolerant Paradighms"
and by National Science Foundation grant number %NSF grant   % National Science Foundation
DMS-1115709       % Temporal multiscale simulation tools for kinetic plasma eqns.
and DMS-1316662.  % Zhengfu's grant
\end{acknowledgements}

\bibliographystyle{unsrt}

% Bibliography
\bibliography{References}

\end{document}

%% file: styles.tex
\newcommand\R{\mathbb{R}}

\newcommand{\Par}[1]{\left(#1\right)}  % parenthesis around arguments

\newcommand{\pd}[2]{\frac{\partial #1}{\partial #2}}            % Jacobian 
\newcommand{\pdn}[3]{{\frac{\partial^{#1}#2}{\partial#3^{#1}}}} % Hessian

\newcommand{\dt}{\Delta t}
\newcommand{\dx}{\Delta x}
\newcommand{\dy}{\Delta y}

\newcommand{\BigOh}{\mathcal{O}}

\newcommand{\eps}{\varepsilon}              % Regulization parameter

\newcommand{\E}{\mathcal{E}}

\newtheorem{rmk}{Remark}%[section]
\newtheorem{lem}{Lemma}%[section]
\newcommand{\hf}{\frac{1}{2}}
\newcommand{\Dt}{\Delta t}
\newcommand{\Dx}{\Delta x}
\newcommand{\Dy}{\Delta y}
\newcommand{\En}{\mathcal{E}}

\providecommand{\norm}[1]{\lVert#1\rVert}